\documentclass[11pt,paper=a4,bibliography=totoc]{scrartcl}

\PassOptionsToPackage{shortlabels}{enumitem}

\usepackage[mathfont=stix2,modernmargins]{mystyle-koma-enhanced}
\usepackage{needspace}

\renewcommand{\P}{\mathbb{P}}
\newcommand{\E}{\mathbb{E}}

\newcommand{\1}{\mathbf{1}}
\newcommand{\cR}{\mathcal{R}}
\newcommand{\cC}{\mathcal{C}}
\newcommand{\cK}{\mathcal{K}}
\newcommand{\Poi}{\operatorname{Poisson}}
\newcommand{\hitorder}{\preceq_{\mathrm{hit}}}
\newcommand{\di}{\,\mathrm{d}}
\DeclareMathOperator{\Aut}{Aut}

\declaretheorem[style=mystyleplain,name=Conjecture 1,numbered=no]{sourceconjecture}
\declaretheorem[style=mystyleplain,name=Theorem A,numbered=no]{theoremA}
\declaretheorem[style=mystyleplain,name=Theorem B,numbered=no]{theoremB}
\declaretheorem[style=mystyleplain,name=Theorem C,numbered=no]{theoremC}

\title{Critical-curve regularity
for finite-lifespan frog models\\ via local-to-global comparisons}

\author{
Christian M\"{o}nch \orcidlink{0000-0002-6531-6482}
\thanks{Independent Researcher, 64289 Darmstadt, Germany}\\
\texttt{cmoench25@gmail.com}
\and
Andreas Klippel
\thanks{Institut f\"ur Mathematik, Johannes Gutenberg-Universit\"at Mainz,
Germany}\\
\texttt{anklippe@uni-mainz.de}
}

\date{September 28, 2026}
\MSC{60K35, 60J10, 82B43}
\Keywords{frog model, local-to-global comparison, dependent percolation, critical curve, sharp phase transition, long-range random walk}

\hypersetup{
  pdftitle={Critical-curve regularity for finite-lifespan frog models via local-to-global comparisons},
  pdfauthor={Christian M\"onch and Andreas Klippel},
  pdfsubject={Probability theory: frog models and critical-curve regularity},
  pdfkeywords={frog model, local-to-global comparison, critical curve, sharpness},
  pdflang={en-GB}
}

\begin{document}
\maketitle

\begin{abstract}
We study the phase boundary of the finite-lifespan frog model.  For rate-one
continuous-time simple random walk on an infinite, connected, locally finite
transitive graph of superlinear growth, we prove that the critical-density
curve is continuous and strictly decreasing, with $-\log\lambda_c$ locally
bi-Lipschitz.  The inverse critical-lifespan curve has the corresponding
regularity wherever it is finite.  Thus the critical-curve conjecture of
Angel, de la Riva, Hermon, and Shi reduces to finiteness of the critical
lifespan at every positive density.  We identify the exact short-lifespan
limit $t\lambda_c(t)\to\Delta\log(1/(1-p_c(G)))$, where $\Delta$ is the
degree and $p_c(G)$ is the bond-percolation threshold, and extend
critical-curve regularity and sharpness to a class of long-range frog models.
The main tool is a local-to-global principle for activation processes
generated by independent finite rooted ranges: local one-hit comparison
implies comparison of global reachability and survival.
The comparison allows disconnected ranges and source-dependent laws, and therefore also applies to activation models beyond random walks on a fixed graph.
\end{abstract}

\section{Introduction}
\label{sec:introduction}

The frog model is built from a simple feedback rule: motion creates more
motion.  Some particles are active and move through a graph, while the others
sleep.  When an active particle reaches a new vertex, the particles waiting
there wake up and start moving as well.  In this way a single initial
activation may set off a cascade that travels through a large part of the
graph.  Early treatments of this type of dynamics go back to Telcs and Wormald
\cite{TelcsWormald1999}.  Since then, the model has been used to study how an
activated region grows through space
\cite{AlvesMachadoPopov2002Shape,RamirezSidoravicius2004}, whether activity
returns to its starting point infinitely often
\cite{HoffmanJohnsonJunge2016,HoffmanJohnsonJunge2017}, and how quickly a
finite graph can be activated
\cite{BenjaminiFontesHermonMachado2020}.  The model resembles the spread of
information or infection, but with an important geometric constraint: the
new particles are already placed throughout the network and can join the
process only after an active particle finds them.

In the classical frog model, an active frog can move forever.  Giving every
frog only a finite lifespan changes the basic question.  Activity can now die
out, so the cascade must create new motion quickly enough to replace what is
lost.  There are two natural ways to help it: place more frogs in the graph,
or allow each frog to walk for longer.  Phase transitions for frog models
with death were already found on Euclidean lattices and regular trees
\cite{AlvesMachadoPopov2002Phase}.  Angel, de la Riva, Hermon, and Shi later
made the tradeoff between the two resources the central object
\cite{AngelDeLaRivaHermonShi2026}.  Their choice of a Poisson number of frogs
and a deterministic walking time gives a particularly clean two-parameter
model: one parameter changes how much dormant activity is present, while the
other changes how far that activity can travel before it disappears.

This formulation leads naturally to a phase transition in terms of particle
density.  Angel et al.\ proved that this transition is sharp on transitive
graphs
\cite{AngelDeLaRivaHermonShi2026}.  Their result belongs to a broad tradition of
sharpness theorems in percolation and related models, see e.g.\
\cite{Menshikov1986,AizenmanBarsky1987,DuminilCopinTassion2016,
DuminilCopinRaoufiTassion2019,DuminilCopinRaoufiTassion2020,
CaiProcacciaZhang2022}.  Sharpness
gives strong information on the two sides of a fixed threshold, but does not
immediately provide insight into how the threshold itself moves when the other
parameter changes.
Understanding that motion is the purpose of the present paper.

Let $G$ be an infinite, connected, locally finite graph.  At every vertex we
place an independent $\Poi(\lambda)$ number of sleeping frogs.  Initially a
fixed vertex $o$ is activated, waking the frogs based there.  Every active
frog performs an independent rate-one continuous-time simple random walk.
When it visits a vertex, all frogs sleeping there wake up and start their own
walks.  Each frog is removed $t$ units of time after it is activated.  We write
$\cC_o^{\lambda,t}$ for the set of vertices activated from $o$, and let $\P$
denote the full product measure governing all frog counts and walks.  The
process survives if this set is infinite, and we write
\[
 \theta(\lambda,t)
 :=\P(|\cC_o^{\lambda,t}|=\infty).
\]
For a fixed lifespan $t$, the critical particle density is
\[
 \lambda_c(t)
 :=\sup\{\lambda\ge0:\theta(\lambda,t)=0\},
\]
while for a fixed density $\lambda$, the critical lifespan is
\[
 t_c(\lambda)
 :=\sup\{t\ge0:\theta(\lambda,t)=0\}.
\]

More frogs and more walking time can only help survival, so both critical
curves are nonincreasing.  Monotonicity alone, however, says very little about
their shape.  A flat piece would mean that extra walking time buys no
reduction in the density needed for survival, while a jump would mark a point where an
arbitrarily small increase in lifespan suddenly lowers that density by a
macroscopic amount.  The real question is therefore whether the exchange
between density and lifespan is genuine and stable at every point.

Angel et al.\ posed this question for transitive graphs, where every vertex
looks the same under graph symmetries.  A transitive graph has
\emph{superlinear growth} if, for one and hence every vertex $o$,
\[
 \limsup_{n\to\infty}\frac{|B_G(o,n)|}{n}=\infty,
\]
where $B_G(o,n)$ is the ball of radius $n$ around $o$.  This is also a natural
geometric threshold in ordinary percolation.  Benjamini and Schramm
conjectured that every quasi-transitive graph of superlinear growth has a
nontrivial Bernoulli percolation threshold
\cite{BenjaminiSchramm1996}; this was proved by Duminil-Copin, Goswami,
Raoufi, Severo, and Yadin
\cite{DuminilCopinGoswamiRaoufiSeveroYadin2020}.  For the finite-lifespan frog
model, Angel et al.\ formulated the following stronger regularity problem
\cite[Conjecture~1.11]{AngelDeLaRivaHermonShi2026}.

\Needspace{9\baselineskip}
\begin{sourceconjecture}
\label{conj:angel-critical-curves}
On every transitive graph of superlinear growth, the maps

\[
 t\longmapsto\lambda_c(t)
 \qquad\text{and}\qquad
 \lambda\longmapsto t_c(\lambda)
\]

are continuous and strictly decreasing.
\end{sourceconjecture}

Our main result settles the regularity of the finite part of this phase
boundary quantitatively and isolates its lower-density endpoint as the only
possible obstruction to the full conjecture.

\Needspace{14\baselineskip}
\begin{theoremA}
\label{thm:introduction-critical-curves}
Let $G$ be an infinite, connected, locally finite, transitive graph of
superlinear growth, and let the frogs perform rate-one continuous-time simple
random walk.  Then
\begin{enumerate}[label=\textup{(\roman*)},widest=iii,
  topsep=0.3em,itemsep=0.2em]
\item $\lambda_c$ is continuous and strictly decreasing on
$(0,\infty)$, and $-\log\lambda_c$ is locally bi-Lipschitz.

\item Setting
\[
 \lambda_\infty:=\lim_{t\to\infty}\lambda_c(t),
\]
the map $t_c$ is finite, continuous, strictly decreasing, and locally
bi-Lipschitz on $(\lambda_\infty,\infty)$.

\item If $\Delta$ is the degree of $G$ and $p_c(G)$ is its
bond-percolation threshold, then
\[
 L_G
 :=\lim_{t\downarrow0}t\lambda_c(t)
 =\lim_{\lambda\to\infty}\lambda t_c(\lambda)
 =\Delta\log\frac1{1-p_c(G)}.
\]
\end{enumerate}
\end{theoremA}

Theorem~A reduces Conjecture~1 to the additional assertion
$\lambda_\infty=0$, or equivalently, finiteness of $t_c(\lambda)$ at every
positive density.  Its three parts are proved in
Theorem~\ref{thm:continuous-critical-density} and
Propositions~\ref{prop:inverse-critical-continuity}
and~\ref{prop:small-lifespan-limit}, respectively.

In \cite{AngelDeLaRivaHermonShi2026} it was shown that $\lambda_\infty=0$ on
nonamenable transitive graphs and on transitive graphs of superlinear
polynomial growth.  Theorem~A therefore resolves Conjecture~1 on these
classes.  Their polynomial-growth argument uses the finite-scale structure
theory of Tessera and Tointon \cite{TesseraTointon2021}; related geometric
ideas also occur in locality results for percolation
\cite{ContrerasMartineauTassion2023}.  We study the endpoint question,
immortal frogs, and quantitative lifespan
bounds in a companion manuscript in preparation.

The short-lifespan constant in Theorem~A(iii) arises because only the first
jump of each frog matters at leading order.  A frog jumps at least once
with probability $1-e^{-t}\sim t$.  Thus, as $t\downarrow0$ with
$\lambda t\to L$, Poisson thinning gives a mean of
$\lambda(1-e^{-t})/\Delta\to L/\Delta$ frogs making their first jump from a
given vertex to a given neighbour.  If each frog is stopped after this
jump, the set of vertices activated from the root has the same law as the
root cluster of independent bond percolation, with parameter tending to
$1-e^{-L/\Delta}$.  Equating this with $p_c(G)$ gives
$L_G=\Delta\log(1/(1-p_c(G)))$.  The expected number of frogs making two or
more jumps is $O(\lambda t^2)=O(t)$ per vertex in this scaling, and the formal proof
uses subcritical susceptibility to show that their contribution does not
change the limiting constant.

Our proofs require comparisons that are not supplied by
monotonicity or by existing sharpness theory.  All frogs born at the same
vertex produce one dependent collection of connections, and a single walk
may connect vertices at arbitrarily large distances.  Standard enhancement
arguments for finite-range percolation
\cite{AizenmanGrimmett1991} therefore do not apply directly.  There is also a
second difficulty.  Strict decrease asks whether extra walking time can
compensate for a loss of frogs.  Continuity asks for control in the opposite
direction: how many extra frogs are needed to compensate for a small loss of
time.  The two parts of the conjecture thus require two comparisons pointing
in opposite directions.

For the forward comparison, we package all paths started from one vertex into
the finite set of vertices that they visit before they die.  We then prove
that a local
ordering of these sets propagates through the whole activation cascade.  This
shows that a definite increase in lifespan compensates for a definite loss
of density and rules out flat pieces.  For the reverse comparison, we work in
a finite region and compare the effect of adding frogs with the effect of
extending their walks.  A bound that is uniform in the region rules out jumps
after passage to the infinite graph.  Together, the two estimates give the
regularity of the finite phase boundary.

We believe that the two comparison results that drive the proof of Theorem~A
are of independent interest.  We begin with the more general one, which
separates the randomness created at a single source from the global cascade
that it may produce.  In fact, the principle can be formulated in a more
general setting.  Let $V$ be a countable set.  On a product probability space
with measure $\P$, let $R=(R_a)_{a\in V}$ and
$\widetilde R=(\widetilde R_a)_{a\in V}$ be two fields of finite random sets,
each having independent coordinates and satisfying
$a\in R_a\cap\widetilde R_a$ almost surely for every $a\in V$.  Given either
field $F=(F_a)_{a\in V}$, draw a directed edge $a\to z$ whenever $z\in F_a$.
For a finite set $\Lambda\subset V$, a vertex $x\in\Lambda$, and a set
$B\subset V\setminus\{x\}$, write $x\rightsquigarrow_\Lambda^F B$ if there is
a directed path from $x$ to $B$ whose non-final vertices lie in $\Lambda$, and
let $\cC_x^F$ be the unrestricted out-cluster of $x$.

\Needspace{15\baselineskip}
\begin{theoremB}
\label{thm:introduction-local-to-global}
Let $R$ and $\widetilde R$ be as above, and assume that, for every $a\in V$
and every finite $A\subset V\setminus\{a\}$,
\[
 \P(R_a\cap A\ne\varnothing)
 \le
 \P(\widetilde R_a\cap A\ne\varnothing).
\]
For every finite $\Lambda\subset V$, every $x\in\Lambda$, and every
$B\subset V\setminus\{x\}$, it then follows that
\[
 \P(x\rightsquigarrow_\Lambda^R B)
 \le
 \P(x\rightsquigarrow_\Lambda^{\widetilde R} B).
\]
Moreover, for every $x\in V$,
\[
 \P(|\cC_x^R|=\infty)
 \le
 \P(|\cC_x^{\widetilde R}|=\infty).
\]
\end{theoremB}

The point of Theorem~B is that it asks much less than a coupling in which one
entire range contains the other, which would correspond to ordinary
stochastic domination.  At a newly reached source, global
reachability needs only one local piece of information: whether the range
hits a suitable target set.  The proof reveals the source ranges one at a
time.  Once all other ranges are fixed, the event under consideration is
either already decided or depends on the exposed range only through such a
one-hit question.  Replacing the ranges successively proves the comparison in
a finite region, and an exhaustion gives the survival statement.

Theorem~B generalises a nearest-neighbour local-to-global comparison of
B\"aumler et al.\ \cite{BaeumlerJahnelKoepplLodewijksReevesTobias2026}.  The
ranges in either field may be disconnected, may contain arbitrarily long
edges, and may have different laws at different sources.  No ambient graph
or symmetry is required.  The result is also related to stochastic comparison
principles developed specifically for frog models \cite{JohnsonJunge2018}
and to Poisson-zoo and random-range percolation models
\cite{RathRokob2022,PeteRokob2025,BetzDrewitzKlippelMoench2026}.  In those
models the sampled sets are usually joined without direction, whereas here direction
matters because a range can be used only after its source has been
activated.

We now return to the frog model.  Here the one-hit probability in Theorem~B
has a simple form because the number of frogs at each source is Poisson.  Let
$Q$ be a Markov kernel on $V$.  Independently at every source $a$, take a
$\Poi(\lambda)$ number of independent rate-one continuous-time $Q$-chains and
let $R_a^{\lambda,t,Q}$ consist of $a$ together with all vertices visited by
these chains by time $t$.  Write
$R^{\lambda,t,Q}=(R_a^{\lambda,t,Q})_{a\in V}$.  For a finite target
$A\subset V\setminus\{a\}$,
\[
 \P(R_a^{\lambda,t,Q}\cap A\ne\varnothing)
 =1-\exp\{-\lambda h_t^Q(a,A)\},
\]
where $h_t^Q(a,A)$ is the probability under $\P$ that a rate-one $Q$-chain
started at $a$ hits $A$ by time $t$.
For this field, abbreviate
$x\rightsquigarrow_\Lambda^{R^{\lambda,t,Q}}B$ to
$x\rightsquigarrow_\Lambda^{\lambda,t,Q}B$, and write
$\cC_x^{\lambda,t,Q}:=\cC_x^{R^{\lambda,t,Q}}$.

If a connected, locally finite reference graph $G$ is given and $Q$ is
equivariant under a transitive subgroup of $\Aut(G)$, fix $o\in V$ and put
\[
 \begin{aligned}
 \theta_Q(\lambda,t)
 &:=\P(|\cC_o^{\lambda,t,Q}|=\infty),\\
 \lambda_c^Q(t)
 &:=\sup\{\lambda\ge0:\theta_Q(\lambda,t)=0\},\\
 t_c^Q(\lambda)
 &:=\sup\{t\ge0:\theta_Q(\lambda,t)=0\}.
 \end{aligned}
\]
For $0<s<t$, set
\[
 \rho(s,t):=\frac{1-e^{-s}}{1-e^{-t}}.
\]

\Needspace{24\baselineskip}
\begin{theoremC}
\label{thm:introduction-clock-comparison}
For every choice of a Markov kernel $Q$ on $V$, parameters $\lambda\ge0$ and
$0<s<t<\infty$, a source $a\in V$, and a finite target
$A\subset V\setminus\{a\}$,
\[
 \P(R_a^{\lambda,s,Q}\cap A\ne\varnothing)
 \le
 \P(R_a^{\rho(s,t)\lambda,t,Q}\cap A\ne\varnothing).
\]
Consequently, for every finite $\Lambda\subset V$, every
$x\in\Lambda$, and every $B\subset V\setminus\{x\}$,
\[
 \P(x\rightsquigarrow_\Lambda^{\lambda,s,Q}B)
 \le
 \P(x\rightsquigarrow_\Lambda^{\rho(s,t)\lambda,t,Q}B),
\]
and, for every $x\in V$,
\[
 \P(|\cC_x^{\lambda,s,Q}|=\infty)
 \le
 \P(|\cC_x^{\rho(s,t)\lambda,t,Q}|=\infty).
\]
In the transitive case,
\[
 \lambda_c^Q(t)
 \le
 \frac{1-e^{-s}}{1-e^{-t}}\lambda_c^Q(s).
\]
Thus $\lambda_c^Q(t)<\lambda_c^Q(s)$ whenever
$0<\lambda_c^Q(s)<\infty$.
\end{theoremC}

Since
\[
 \rho(s,t)=\frac{1-e^{-s}}{1-e^{-t}}<1,
 \qquad 0<s<t,
\]
Theorems~B and~C say that a longer lifespan compensates for a fixed fraction
of the frogs, not merely for an arbitrarily small loss of density.  The
factor is uniform over the source, the target, and the number of jumps needed
to reach it.  On the finite, nontrivial part of the phase boundary, this is
what turns ordinary monotonicity into strict decrease, even for kernels with
unbounded jump length.  The same comparison shows that
$(1-e^{-t})\lambda_c^Q(t)$ is nonincreasing.  For simple random walk, a
first-jump comparison with Bernoulli bond percolation gives a matching
upper bound for the short-lifespan constant.  Controlling the contribution
of frogs making two or more jumps proves the exact limit in Theorem~A.

Theorem~C controls the boundary in only one direction.  Where the critical
density is positive and finite, it rules out flat pieces, but it does not rule
out a downward jump.  Continuity requires a reverse comparison.  In a finite
region, changing the density means adding frogs, whereas changing the
lifespan means extending their walks.  We use finite-volume sensitivity
formulas to measure both effects and combine them with a geometric exit
estimate from Angel et al. \cite{AngelDeLaRivaHermonShi2026}.  This bounds the
effect of extra time by the effect of extra density, uniformly in the region.
After passing to the infinite graph, the reverse estimate combines with
Theorem~C to give continuity and the local bi-Lipschitz bounds in Theorem~A.

The local-to-global theorem and the forward clock comparison do not use
nearest-neighbour geometry.  Only the reverse estimate does.  This suggests
asking which simple property of a general jump kernel can replace the
geometric estimate for simple random walk.  We show that it is enough to take
an irreducible kernel $Q$, equivariant under a transitive subgroup of
$\Aut(G)$, for which there is a constant $\kappa<\infty$ such that
\[
 Q^2(x,y)\le\kappa Q(x,y),
 \qquad x\ne y.
\]
In words, reaching a different vertex in two jumps cannot be much more likely
than reaching it in one direct jump.  On every infinite, connected, locally
finite, transitive graph of superlinear growth, this condition gives the same
regularity of $\lambda_c^Q$ and of the finite branch of $t_c^Q$ as in
Theorem~A.

We also adapt the sharpness argument to this long-range setting.  Immediately
above the critical density, the survival probability has a linear lower
bound; below it, the activated cluster has an exponential size tail.  A main
example is a symmetric, translation-invariant power-law kernel on a
polynomial-growth Cayley graph: if
\[
 |B_G(o,n)|\asymp n^d,
 \qquad
 Q(x,y)\asymp(1+d_G(x,y))^{-(d+\alpha)},
 \qquad x\ne y,\quad d>1,\quad\alpha>0,
\]
then both critical curves are continuous and strictly decreasing on all of
$(0,\infty)$.  These results complement recent work on activation times and
finite-volume cover lifespans for heavy-tailed frog models by Angel, Hermon, and Shi \cite{AngelHermonShi2026HeavyTailed}.

The paper is organised as follows.  Section~\ref{sec:rooted-ranges} proves the
local-to-global theorem.  Section~\ref{sec:forward-comparison} derives the
forward density--lifespan comparison, and
Section~\ref{sec:reverse-comparison} proves the reverse comparison.
Section~\ref{sec:critical-curves} combines the two directions to establish
the critical-curve results and the small-lifespan limit.  Finally,
Section~\ref{sec:bridge-kernels} treats the long-range extension, while the
appendix gives the corresponding adaptation of the sharpness argument from \cite{AngelDeLaRivaHermonShi2026}.

\section{Rooted activation ranges and local-to-global comparison}
\label{sec:rooted-ranges}

Let $V$ be countable.  A rooted-range field itself need not use an ambient graph.
Whenever we do invoke a reference graph, we mean a deterministic connected
locally finite graph $G=(V,E_G)$ that is used only for geometry and symmetry, in particular the directed activation edges need not belong to $E_G$.

\begin{definition}
\label{def:rooted-range}
For $a\in V$, set

\[
 \cR_a:=\{R\subset V:a\in R,\ |R|<\infty\}.
\]

A \emph{rooted-range law at $a$} is a probability measure on $\cR_a$.
\end{definition}

Given a field $(R_a)_{a\in V}$ of rooted activation ranges, draw the
directed macro-edge $a\to z$ whenever $z\in R_a$.  We always assume that
the ranges belonging to distinct source vertices are independent.

\begin{definition}
\label{def:one-hit}
For rooted-range laws $\mu_a$ and $\nu_a$ on $\cR_a$, write
$\mu_a\hitorder\nu_a$ if

\[
 \mu_a(R\cap A\ne\varnothing)
 \le \nu_a(R\cap A\ne\varnothing)
\]

for every finite $A\subset V\setminus\{a\}$.
\end{definition}

The one-hit order is weaker than stochastic domination by inclusion.  It
retains exactly the local information used by a directed path when it leaves
a source vertex.

\begin{proposition}
\label{prop:subgraph-reduction}
Suppose that activating a source $a$ triggers a finite directed subgraph
$H_a$ rooted at $a$ and thereby activates every vertex reachable from $a$
within $H_a$.  The edges of $H_a$ are used only through this
root-triggered operation.  Put
\[
 R_a:=\operatorname{Reach}_{H_a}(a)
 :=\{z\in V:a\text{ reaches }z\text{ in }H_a\}.
\]
Replacing every local subgraph $H_a$ by its rooted activation range $R_a$
preserves the eventual activated set.
\end{proposition}

\begin{proof}
Once $a$ is activated, the local subgraph $H_a$ eventually activates
precisely the vertices in $\operatorname{Reach}_{H_a}(a)$.  Replacing $H_a$
by macro-edges
from $a$ to these vertices therefore leaves the one-step activation closure
of every set unchanged.  Iterating the closure operation gives the same
eventual activated set.
\end{proof}

\begin{definition}
\label{def:volume-reachability}
Let $\Lambda\subset V$ be finite, let $x\in\Lambda$, and let
$B\subset V\setminus\{x\}$.  We say that $x$ \emph{reaches $B$ through
$\Lambda$}, and write $x\rightsquigarrow_\Lambda B$, if there are
vertices $x=x_0,x_1,\ldots,x_n$ such that

\[
 x_0,\ldots,x_{n-1}\in\Lambda,
 \qquad x_n\in B,
 \qquad x_{i+1}\in R_{x_i}\quad(0\le i<n).
\]

When saying that a vertex reaches itself, we use the
empty-path convention.  Only the sources of the non-final steps must
lie in $\Lambda$.  Thus the
proper finite-volume exit event for a long-range model is
$x\rightsquigarrow_\Lambda(V\setminus\Lambda)$, rather than connection to a
nearest-neighbour boundary.
\end{definition}

\begin{lemma}
\label{lem:arbitrary-targets}
If $\mu_a\hitorder\nu_a$, then the inequality in
Definition~\ref{def:one-hit} holds for every possibly infinite set
$A\subseteq V\setminus\{a\}$.
\end{lemma}

\begin{proof}
Since $V$ is countable, choose finite sets $A_n\uparrow A$.  Under either
law, the events \(\{R\cap A_n\ne\varnothing\}\) increase to
$\{R\cap A\ne\varnothing\}$.  The conclusion follows by continuity from
below.
\end{proof}

\begin{lemma}
\label{lem:single-site-section}
Fix a finite $\Lambda\subset V$, a source $x\in\Lambda$, a target
$B\subset V\setminus\{x\}$, and a site $a\in\Lambda$.  Conditional on
$(R_z)_{z\ne a}$, the event $\{x\rightsquigarrow_\Lambda B\}$, viewed as
a function of $R_a$, is either constant or has the form

\[
 \{R_a\cap A_a\ne\varnothing\}
\]

for a set $A_a\subset V\setminus\{a\}$ determined by the conditioned
ranges.
\end{lemma}

\begin{proof}
We delete every outgoing activation edge from $a$, and call the resulting
directed configuration $\mathcal D^{-a}$.  If $x$ reaches $B$ through
$\Lambda$ in $\mathcal D^{-a}$, the section is constant one.  If it does
not and $x$ does not reach $a$ in $\mathcal D^{-a}$, the section is
constant zero.  Indeed, the first visit to $a$ by any witnessing path would
have to occur without using an outgoing edge from $a$.

It remains to consider the case in which $x$ reaches $a$ but not $B$ in
$\mathcal D^{-a}$.  Define

\[
 A_a:=(B\setminus\{a\})
 \cup
 \{z\in\Lambda\setminus(B\cup\{a\}):
      z\rightsquigarrow_\Lambda B\text{ in }\mathcal D^{-a}\}.
\]

If $R_a$ hits $A_a$, concatenate a walk from $x$ to $a$, the edge from
$a$ to the hit vertex, and a walk from that vertex to $B$.
Conversely, suppose a witnessing path exists after the outgoing edges from
$a$ are restored.  Delete any self-loop steps from the witness, since they do
not affect reachability, and consider its last step whose source is $a$.
Such a step exists in the present case, and its endpoint is different from
$a$.  That endpoint lies in $A_a$, because the remaining part of the walk
uses no outgoing edge from $a$.  Hence, with all other ranges fixed,
$x\rightsquigarrow_\Lambda B$ holds if and only if
$R_a\cap A_a\ne\varnothing$.
\end{proof}

\begin{theorem}
\label{thm:local-to-global}
For every $a\in V$, let $\mu_a$ and $\nu_a$ be rooted-range laws on
$\cR_a$ satisfying $\mu_a\hitorder\nu_a$.  Let $\P_\mu$ and
$\P_\nu$ denote the product laws induced by these families.  Then, for every
finite
$\Lambda\subset V$, every $x\in\Lambda$, and every
$B\subset V\setminus\{x\}$,

\[
 \P_\mu(x\rightsquigarrow_\Lambda B)
 \le
 \P_\nu(x\rightsquigarrow_\Lambda B).
\]
\end{theorem}

\begin{proof}
Only the laws at the finitely many source sites in $\Lambda$ affect the
event.  Enumerate those sites and interpolate between the two product fields
by replacing one local law at a time.  At a replacement site $a$, condition
on all other ranges and apply Lemma~\ref{lem:single-site-section}.  If the
connection event does not depend on $R_a$, replacing its law leaves the
conditional probability unchanged.  Otherwise, the event occurs exactly
when $R_a$ hits a target $A_a$ determined by the other ranges.  The comparison
then follows from Definition~\ref{def:one-hit} and, when $A_a$ is infinite,
Lemma~\ref{lem:arbitrary-targets}.  Integrate
over the other ranges and iterate through the sites of $\Lambda$.
\end{proof}

\begin{remark}
\label{rem:relation-baeumler}
When $V=\mathbb Z^d$ and every $R_a\setminus\{a\}$ is a subset of the
nearest neighbours of $a$, Theorem~\ref{thm:local-to-global} recovers the
directed-reachability comparison of
\cite[Theorem~2.1]{BaeumlerJahnelKoepplLodewijksReevesTobias2026}.  The
proof of Theorem~\ref{thm:local-to-global} uses neither transitivity,
connectedness of a range, nor local finiteness of the activation relation.
A reference graph is needed only for later geometric or symmetry questions,
such as balls, growth, boundaries, or automorphisms.
\end{remark}

\begin{corollary}
\label{cor:infinite-volume-comparison}
Let $\cC_x$ be the out-cluster of $x$.  Under the assumptions of
Theorem~\ref{thm:local-to-global}, for every $x\in V$,

\[
 \P_\mu(|\cC_x|=\infty)
 \le
 \P_\nu(|\cC_x|=\infty).
\]
\end{corollary}

\begin{proof}
Since $V$ is countable, choose an increasing sequence of finite sets
$\Lambda_n\uparrow V$ with $x\in\Lambda_1$, and let

\[
 E_n:=\{x\rightsquigarrow_{\Lambda_n}(V\setminus\Lambda_n)\}.
\]

The events $E_n$ decrease in $n$: if $m\ge n$ and a walk witnesses $E_m$,
then its initial segment up to its first exit from $\Lambda_n$ witnesses
$E_n$.  Moreover,

\[
 \bigcap_{n\ge1}E_n=\{|\cC_x|=\infty\},
\]

because a finite out-cluster is eventually contained in $\Lambda_n$, whereas
an infinite out-cluster exits every $\Lambda_n$.
Apply Theorem~\ref{thm:local-to-global} to each $E_n$ and then use
continuity from above.
\end{proof}

\section{Finite-lifespan frog models and the forward comparison}
\label{sec:forward-comparison}

We begin by encoding the frog model as a rooted-range field.

Let $Q$ be a Markov kernel on the countable set $V$, and write $\P_x^Q$ for
the law of a rate-one continuous-time $Q$-chain started at $x$.

Concretely, set $X_u=Y_{N_u}$, where $(Y_n)$ has transition kernel $Q$ and
$(N_u)$ is an independent rate-one Poisson process.  Throughout, a clock
update (also called a jump) includes an update with $Y_{n+1}=Y_n$.
Thus $N_u\sim\Poi(u)$ also when $Q$ has a holding mass.
  For every
$x\in V$, let $\eta_x\sim\Poi(\lambda)$, independently over $x$.  Conditional
on $\eta_x$, sample independent rate-one continuous-time $Q$-walks
$(X_s^{x,i})_{s\ge0}$, $1\le i\le\eta_x$, started at $x$.  For a fixed
finite lifespan $t$, define

\begin{equation}
 R_x^{\lambda,t,Q}
 :=\{x\}\cup
 \bigcup_{i=1}^{\eta_x}\{X_s^{x,i}:0\le s\le t\}.
 \label{eq:frog-range}
\end{equation}

Since a rate-one clock rings only finitely many times on a finite interval,
this is a finite rooted activation range almost surely, even when $Q$ has
infinite range.  The resulting field of rooted-range laws is denoted by
$\mu_{\lambda,t,Q}$, and its induced product measure by
$\P_{\lambda,t,Q}$.  For two such fields, we use $\hitorder$ sitewise, meaning
that the order in Definition~\ref{def:one-hit} holds at every source.

For $A\subseteq V$, denote
\[
 \tau_A:=\inf\{s\ge0:X_s\in A\},
\]
with $\tau_A=\infty$ when the set is never hit.

\begin{proposition}
\label{prop:frog-one-hit}
For every $x\in V$ and every finite $A\subset V\setminus\{x\}$,

\begin{equation}
 \P(R_x^{\lambda,t,Q}\cap A\ne\varnothing)
 =1-\exp\{-\lambda h_t^Q(x,A)\},
 \qquad
 h_t^Q(x,A):=\P_x^Q(\tau_A\le t).
 \label{eq:poisson-hit-functional}
\end{equation}

In particular,

\[
 \mu_{\lambda,t,Q}\hitorder\mu_{\lambda',t',Q'}
\]

if and only if

\[
 \lambda h_t^Q(x,A)
 \le \lambda' h_{t'}^{Q'}(x,A)
\]

for every source $x$ and every finite target
$A\subset V\setminus\{x\}$.
\end{proposition}

\begin{proof}
Each frog at $x$ hits $A$ by time $t$ with probability
$h_t^Q(x,A)$.  Poisson thinning shows that the number of such frogs is
$\Poi(\lambda h_t^Q(x,A))$.  Its probability of being nonzero is the
right-hand side of \eqref{eq:poisson-hit-functional}.  The equivalence in
Proposition~\ref{prop:frog-one-hit} follows because
$u\mapsto1-e^{-u}$ is strictly increasing.
\end{proof}

When a vertex is activated in the frog model, it exposes precisely the range
in \eqref{eq:frog-range}.  Iterating this rule shows that the set of eventually
activated vertices is the out-cluster in the rooted-range model.  For simple
random walk this is the percolation representation used in
\cite{AngelDeLaRivaHermonShi2026}.

\subsection*{Jump-budget comparison}

Let $(Y_n)_{n\ge0}$ be the embedded discrete-time chain with kernel $Q$,
started at $x$.  A nonnegative integer-valued random variable $L$,
independent of the chain, will be called a \emph{jump budget}; the associated
single-particle range is \(\{Y_0,\ldots,Y_L\}\).
For a target $A\subset V\setminus\{x\}$, write

\[
 H_A:=\inf\{n\ge1:Y_n\in A\},
\]

with $H_A=\infty$ if the target is never hit.

\begin{lemma}
\label{lem:jump-budget-comparison}
Let $L$ and $L'$ be jump budgets.  If, for some $\rho\ge0$,

\[
 \P(L\ge n)\le\rho\P(L'\ge n),\qquad n\ge1,
\]

then, for every $A\subset V\setminus\{x\}$,

\[
 \P_x^Q(H_A\le L)
 \le \rho\P_x^Q(H_A\le L').
\]
\end{lemma}

\begin{proof}
Independence gives

\[
 \P_x^Q(H_A\le L)
 =\sum_{n\ge1}\P_x^Q(H_A=n)\P(L\ge n).
\]

Apply the assumed tail inequality term by term.  The series identity remains
valid for infinite $A$.
\end{proof}

\begin{corollary}
\label{cor:general-jump-budgets}
For a jump budget $J$ and $\lambda\ge0$, let $\mu_{\lambda,J,Q}$ be the
rooted-range field generated at every source by a $\Poi(\lambda)$ number of
independent $Q$-chains, each run for an independent copy of $J$ steps, and
write $\P_{\lambda,J,Q}$ for its product law.  If $\lambda,\lambda'\ge0$ and
\[
 \lambda\P(J\ge n)\le\lambda'\P(J'\ge n),
 \qquad n\ge1,
\]
then
\[
 \mu_{\lambda,J,Q}\hitorder\mu_{\lambda',J',Q}.
\]
Consequently, all finite-volume reachability probabilities and
infinite-out-cluster probabilities are ordered in the same direction.  More
generally, the same conclusion holds for source-dependent data whenever
\[
 \lambda_x\P(J_x\ge n)
 \le\lambda'_x\P(J'_x\ge n)
 \qquad\text{for every source $x$ and every $n\ge1$}.
\]

Suppose now that the reference graph is infinite, locally finite, and
transitive, and that $Q$ is equivariant under a transitive subgroup of its
automorphism group.  Define
\[
 \lambda_c^{J,Q}
 :=\sup\{\lambda\ge0:
          \P_{\lambda,J,Q}(|\cC_o|=\infty)=0\}.
\]
If $\rho>0$ and
\[
 \P(J\ge n)\le\rho\P(J'\ge n),
 \qquad n\ge1,
\]
then
\[
 \lambda_c^{J',Q}\le\rho\lambda_c^{J,Q}.
\]
If additionally $0<\rho<1$ and
$0<\lambda_c^{J,Q}<\infty$, then
$\lambda_c^{J',Q}<\lambda_c^{J,Q}$.
\end{corollary}

\begin{proof}
For a target $A\subset V\setminus\{x\}$,
\[
 p_J^Q(x,A):=\P_x^Q(H_A\le J)
 =\sum_{n\ge1}\P_x^Q(H_A=n)\P(J\ge n).
\]
Thus the assumed inequalities imply
$\lambda p_J^Q(x,A)\le\lambda'p_{J'}^Q(x,A)$.  Poisson thinning followed by
Theorem~\ref{thm:local-to-global} and
Corollary~\ref{cor:infinite-volume-comparison} proves the reachability
claims, source by source in the inhomogeneous case.  For the critical
comparison, if $\lambda_c^{J,Q}=\infty$, the claimed inequality is
tautological in the extended nonnegative reals because $\rho>0$.  Otherwise
particle-density monotonicity makes the zero-survival set downward closed, so
every $u>\lambda_c^{J,Q}$ is supercritical.  Apply the first part at
densities $u$ and $\rho u$, then let
$u\downarrow\lambda_c^{J,Q}$.  This also covers
$\lambda_c^{J,Q}=0$.  If in addition $0<\rho<1$ and
$0<\lambda_c^{J,Q}<\infty$, then
\[
 \lambda_c^{J',Q}
 \le\rho\lambda_c^{J,Q}<\lambda_c^{J,Q},
\]
which gives the final strict inequality.  The critical-value statement
deliberately excludes $\rho=0$: its tail hypothesis then forces $J=0$
almost surely and hence $\lambda_c^{J,Q}=\infty$, so there is no meaningful
finite rescaling bound to assert.
\end{proof}

\begin{lemma}
\label{lem:poisson-clock-contraction}
Let $N_u\sim\Poi(u)$.  For $0<s<t<\infty$,

\begin{equation}
 \P(N_s\ge n)
 \le
 \frac{1-e^{-s}}{1-e^{-t}}\P(N_t\ge n),
 \qquad n\ge1.
 \label{eq:poisson-clock-contraction}
\end{equation}

The constant is strictly smaller than one and is optimal uniformly in
$n\ge1$.
\end{lemma}

\begin{proof}
Let $S_n$ be the time of the $n$-th event of a rate-one Poisson process and
put

\[
 F_n(u):=\P(S_n\le u)=\P(N_u\ge n).
\]

Since $S_n$ has the $\operatorname{Gamma}(n,1)$ density,

\[
 \frac{F_n(u)}{F_1(u)}
 =\frac{1}{(n-1)!}
   \frac{\int_0^u r^{n-1}e^{-r}\di r}
        {\int_0^u e^{-r}\di r}.
\]

The ratio $F_n(u)/F_1(u)$ is nondecreasing in $u$.  Its derivative has the
sign of

\[
 u^{n-1}\int_0^u e^{-r}\di r
 -\int_0^u r^{n-1}e^{-r}\di r,
\]

which is nonnegative because $r^{n-1}\le u^{n-1}$ on $[0,u]$.  Therefore

\[
 \frac{F_n(s)}{F_n(t)}
 \le\frac{F_1(s)}{F_1(t)}
 =\frac{1-e^{-s}}{1-e^{-t}}.
\]

Equality holds for $n=1$, which proves optimality.
\end{proof}

The number of clock updates of a rate-one continuous-time
$Q$-walk by time $u$ is $N_u$, independently of the embedded chain.  Combining
Lemmas~\ref{lem:jump-budget-comparison} and
\ref{lem:poisson-clock-contraction} gives, for every target
$A\subset V\setminus\{x\}$,

\begin{equation}
 h_s^Q(x,A)
 \le \rho(s,t)h_t^Q(x,A),
 \qquad
 \rho(s,t)=\frac{1-e^{-s}}{1-e^{-t}}.
 \label{eq:clock-hit-comparison}
\end{equation}

\subsection*{Strict comparison}

\begin{proposition}
\label{prop:density-lifespan-tradeoff}
For every Markov kernel $Q$, every $\lambda\ge0$, and $0<s<t<\infty$,

\[
 \mu_{\lambda,s,Q}\hitorder
 \mu_{\rho(s,t)\lambda,t,Q}.
\]

Consequently, finite-volume reachability probabilities under
$(\lambda,s,Q)$ are no larger than under $(\rho(s,t)\lambda,t,Q)$.  The
infinite-out-cluster probability under $(\lambda,s,Q)$ is also no larger than
under $(\rho(s,t)\lambda,t,Q)$.
\end{proposition}

\begin{proof}
For every finite $A\subset V\setminus\{x\}$,
Proposition~\ref{prop:frog-one-hit} and
\eqref{eq:clock-hit-comparison} give

\begin{align*}
 \P(R_x^{\lambda,s,Q}\cap A\ne\varnothing)
 &=1-\exp\{-\lambda h_s^Q(x,A)\}\\
 &\le1-\exp\{-\rho(s,t)\lambda h_t^Q(x,A)\}\\
 &=\P(R_x^{\rho(s,t)\lambda,t,Q}\cap A\ne\varnothing).
\end{align*}

Now apply Theorem~\ref{thm:local-to-global} and
Corollary~\ref{cor:infinite-volume-comparison}.
\end{proof}

Assume now that $G$ is an infinite, locally finite, transitive reference
graph, that $Q$ is equivariant under a transitive subgroup of
$\Aut(G)$, and that $o\in V$ is fixed.  We set

\[
 \begin{aligned}
 \theta_Q(\lambda,t)
 &:=\P_{\lambda,t,Q}(|\cC_o|=\infty),
 \qquad \theta_Q(\lambda,0):=0,\\
 \lambda_c^Q(t)
 &:=\sup\{\lambda\ge0:\theta_Q(\lambda,t)=0\}.
 \end{aligned}
\]

\begin{theorem}
\label{thm:strict-critical-density}
For every $0<s<t<\infty$,

\begin{equation}
 \lambda_c^Q(t)
 \le
 \frac{1-e^{-s}}{1-e^{-t}}\lambda_c^Q(s).
 \label{eq:critical-density-comparison}
\end{equation}

In particular, if $0<\lambda_c^Q(s)<\infty$, then
$\lambda_c^Q(t)<\lambda_c^Q(s)$.
\end{theorem}

\begin{proof}
Proposition~\ref{prop:density-lifespan-tradeoff} gives

\[
 \theta_Q(\lambda,s)
 \le \theta_Q(\rho(s,t)\lambda,t),
 \qquad \lambda\ge0.
\]

If $\lambda_c^Q(s)=\infty$, \eqref{eq:critical-density-comparison} is
immediate in the
extended nonnegative reals.  Otherwise, take
$\lambda>\lambda_c^Q(s)$.  Monotonicity in $\lambda$ and the definition of
the critical value imply $\theta_Q(\lambda,s)>0$.  Consequently,
$\lambda_c^Q(t)\le\rho(s,t)\lambda$.  Let
$\lambda\downarrow\lambda_c^Q(s)$.  If
$0<\lambda_c^Q(s)<\infty$, then $\rho(s,t)<1$ also gives
$\lambda_c^Q(t)<\lambda_c^Q(s)$.
\end{proof}

Proposition~\ref{prop:density-lifespan-tradeoff} also has an infinitesimal
form.  For finite $\Lambda\ni o$, put
$f_\Lambda^Q(\lambda,t):=\P_{\lambda,t,Q}(o\rightsquigarrow_\Lambda\Lambda^c)$. Proposition~\ref{prop:russo-all-kernels} below shows that the latter probability is
$C^1$ on the positive parameter quadrant.  For fixed $k>0$,
Proposition~\ref{prop:density-lifespan-tradeoff} says that

\[
 t\longmapsto
 f_\Lambda^Q\left(\frac{k}{1-e^{-t}},t\right)
\]

is nondecreasing.  The chain rule therefore yields

\begin{equation}
 \partial_t f_\Lambda^Q(\lambda,t)
 \ge
 \frac{\lambda}{e^t-1}\,
 \partial_\lambda f_\Lambda^Q(\lambda,t).
 \label{eq:forward-differential-comparison}
\end{equation}

At the critical-parameter level,
$(1-e^{-t})\lambda_c^Q(t)$ is nonincreasing.

\begin{remark}[The long-range case]

Theorem~\ref{thm:local-to-global} applies to a deterministic finite
jump budget $\ell$.  In that case one replaces \eqref{eq:frog-range} by

\[
 R_x^{\lambda,\ell,Q}
 :=\{x\}\cup
 \bigcup_{i=1}^{\eta_x}\{Y_0^{x,i},\ldots,Y_\ell^{x,i}\},
\]

and obtains

\[
 \P(R_x^{\lambda,\ell,Q}\cap A\ne\varnothing)
 =1-\exp\{-\lambda\P_x^Q(H_A\le\ell)\}.
\]

For a heavy-tailed kernel on $\mathbb Z^d$ with $d\ge2$, for example

\[
 Q(x,y)\asymp |x-y|^{-(d+\alpha)},\qquad x\ne y,
\]

the range is finite but need not be connected in the nearest-neighbour
reference graph, so Theorem~\ref{thm:local-to-global} applies.  Such
finite-step ranges occur in \cite{AngelHermonShi2026HeavyTailed}.  The
corresponding continuous-time Poissonisation is covered by
Corollary~\ref{cor:heavy-tailed-critical-curves}. Its critical densities are
positive and finite, and Proposition~\ref{prop:density-lifespan-tradeoff}
strictly orders them as the lifespan varies. Note that for different values of
values of $\alpha$, the hitting probabilities need not be ordered for all targets, so this case is in general not covered.
\end{remark}

\section{The reverse Russo comparison}
\label{sec:reverse-comparison}

Let $G$ be an infinite, connected, locally finite, transitive graph of degree
$\Delta$, and let $Q_G$ be its simple-random-walk kernel.  Throughout this
section, abbreviate
\[
 \P_{\lambda,t}:=\P_{\lambda,t,Q_G},
 \qquad
 \P_x:=\P_x^{Q_G}.
\]
Fix a finite set $\Lambda\ni o$ and write

\begin{equation}
 f_\Lambda(\lambda,t)
 :=\P_{\lambda,t}
 \bigl(o\rightsquigarrow_\Lambda\Lambda^c\bigr).
 \label{eq:finite-volume-exit}
\end{equation}

Here paths are stopped at their first visit to $\Lambda^c$, so every
non-final source lies in $\Lambda$, as required in
Definition~\ref{def:volume-reachability}.

Set

\begin{equation}
 \theta(\lambda,t):=\P_{\lambda,t}(|\cC_o|=\infty).
 \label{eq:survival-probability}
\end{equation}

We set $\theta(\lambda,0):=0$, consistently with the model at zero lifespan.

\subsection*{Russo decompositions and the exit estimate}

For a finite $S\ni o$, retain at each $y\in S$ only those frogs whose
entire time-$t$ traces stay in $S$.  Let

\begin{equation}
 a_{\lambda,t}^S(x)
 :=\P_{\lambda,t}
 \bigl(o\text{ activates }x\text{ using only the retained frogs}\bigr),
 \qquad x\in S.
 \label{eq:internal-activation-weight}
\end{equation}

For an independent test walk started at $x\in S$, let

\begin{equation}
 h_t^S(x):=\P_x(\tau_{S^c}\le t),
 \qquad
 q_t^S(x):=\frac{\mathrm d}{\mathrm d t}h_t^S(x),
 \label{eq:exit-probability-density}
\end{equation}

and let $N_t$ be its number of jumps by time $t$.
The derivative exists and is continuous for $t>0$ by the Gamma-mixture
representation in Lemma~\ref{lem:exit-density} below.
Finally, let

\begin{equation}
 L=L_\Lambda
 :=\{x\in\Lambda:x\not\rightsquigarrow_\Lambda\Lambda^c\}.
 \label{eq:nonexit-set}
\end{equation}

\begin{proposition}
\label{prop:russo-decompositions}
For every $\lambda,t>0$,

\begin{align}
 \partial_\lambda f_\Lambda(\lambda,t)
 &=\sum_{\substack{S\subseteq\Lambda\\o\in S}}
   \P_{\lambda,t}(L=S)
   \sum_{x\in S}a_{\lambda,t}^S(x)h_t^S(x),
 \label{eq:russo-lambda}\\
 \partial_t f_\Lambda(\lambda,t)
 &=\lambda\sum_{\substack{S\subseteq\Lambda\\o\in S}}
   \P_{\lambda,t}(L=S)
   \sum_{x\in S}a_{\lambda,t}^S(x)q_t^S(x).
 \label{eq:russo-time}
\end{align}
\end{proposition}

\begin{proof}
Apply the self-contained kernel-independent
Proposition~\ref{prop:russo-all-kernels} below with $Q$ equal to simple
random walk.  Its proof also shows that $f_\Lambda$ is $C^1$ on the positive
parameter quadrant.  The resulting identities agree with
\cite[Equations~(3.5) and~(3.10)]{AngelDeLaRivaHermonShi2026}.
\end{proof}

\begin{lemma}
\label{lem:exit-weighted-jumps}
For every $\lambda,t>0$, there is a finite constant
$C(\Delta,\lambda,t)>0$, continuous in $(\lambda,t)$, such that for every
finite $S\ni o$,

\begin{align}
 &\sum_{x\in S}a_{\lambda,t}^S(x)h_t^S(x)
   \E_x[N_t\mid\tau_{S^c}\le t]
 \notag\\
 &\hspace{35mm}\le
 C(\Delta,\lambda,t)
 \sum_{x\in S}a_{\lambda,t}^S(x)h_t^S(x).
 \label{eq:exit-weighted-jumps}
\end{align}
\end{lemma}

\begin{proof}
For connected $S$, inequality \eqref{eq:exit-weighted-jumps}, after
multiplication by $\lambda$, is
\cite[Equation~(3.12) in Lemma~3.3]{AngelDeLaRivaHermonShi2026}.  The
reciprocal of the constant in that lemma is continuous and strictly positive
on the positive parameter quadrant, so $C$ is finite and continuous there.

\cite[Lemma~3.3]{AngelDeLaRivaHermonShi2026} is stated for connected $S$.
For a general finite $S\ni o$,
let $S_o$ be the reference-graph component of $o$ in $S$.  Internal
nearest-neighbour activation from $o$ cannot reach a vertex in
$S\setminus S_o$, so $a_{\lambda,t}^S(x)=0$ there.  For $x\in S_o$, a
nearest-neighbour trace started in $S_o$ stays in $S$ if and only if it stays
in $S_o$: a first visit to another component of $S$ would have to cross a
vertex of $S^c$.  Consequently
$a_{\lambda,t}^S(x)=a_{\lambda,t}^{S_o}(x)$, and the first exit times from
$S_o$ and from $S$ agree pathwise for walks started in $S_o$.  Thus both
sides of \eqref{eq:exit-weighted-jumps} reduce to the connected set $S_o$,
to which \cite[Lemma~3.3]{AngelDeLaRivaHermonShi2026} applies.
\end{proof}

\Needspace{10\baselineskip}
We next relate the exit-time density in the lifespan derivative to the number
of jumps made by an exiting walk.  Let $(Y_n)_{n\ge0}$ be the embedded simple
random walk and, for $k\ge1$, set

\[
 H_S:=\inf\{n\ge1:Y_n\notin S\},
 \qquad
 g_k(t):=\frac{t^{k-1}e^{-t}}{(k-1)!}.
\]

Thus $g_k$ is the density of a $\operatorname{Gamma}(k,1)$ random variable.

\begin{lemma}
\label{lem:exit-density}
For every finite $S\ni x$ and every $t>0$,
\begin{equation}
 t q_t^S(x)
 =\E_x\bigl[N_t\1_{\{\tau_{S^c}\le t\}}\bigr]-t h_t^S(x).
 \label{eq:exit-density-identity}
\end{equation}
In particular,
\begin{equation}
 t q_t^S(x)
 \le
 \E_x\bigl[N_t\1_{\{\tau_{S^c}\le t\}}\bigr]
 =h_t^S(x)\E_x[N_t\mid\tau_{S^c}\le t].
 \label{eq:exit-density-bound}
\end{equation}
\end{lemma}

\begin{proof}
Because $G$ is infinite and connected while $S$ is finite, there is a finite
nearest-neighbour path from $x$ to $S^c$.  Hence $h_t^S(x)>0$ for every
$t>0$, so the conditional expectation in
\eqref{eq:exit-density-bound} is well defined.

The embedded chain and the Poisson clock are independent, and the exit event
is $\{H_S\le N_t\}$.  We condition on $H_S$ to obtain
\begin{equation}
 \begin{aligned}
  h_t^S(x)&=\sum_{k\ge1}\P_x(H_S=k)\P(N_t\ge k),\\
  q_t^S(x)&=\sum_{k\ge1}\P_x(H_S=k)g_k(t).
 \end{aligned}
 \label{eq:exit-time-mixtures}
\end{equation}
To justify differentiating the mixture locally uniformly, fix
$0<a<b<\infty$.  For $u\in[a,b]$,
\begin{equation}
 0\le\P_x(H_S=k)g_k(u)
 \le e^{-a}\frac{b^{k-1}}{(k-1)!},
 \label{eq:gamma-mixture-domination}
\end{equation}
and the majorant is summable over $k\ge1$.

For every $k\ge1$, the Poisson identity
$\E[N_t\1_{\{N_t\ge k\}}]=t\P(N_t\ge k-1)$ gives
\begin{equation}
 \begin{aligned}
 t g_k(t)
 &=k\P(N_t=k)=t\P(N_t=k-1)\\
 &=\E[N_t\1_{\{N_t\ge k\}}]-t\P(N_t\ge k).
 \end{aligned}
 \label{eq:gamma-poisson-pointwise}
\end{equation}
Now multiply by $\P_x(H_S=k)$ and sum over $k$.  Independence and
\eqref{eq:exit-time-mixtures} yield \eqref{eq:exit-density-identity}.
Both nonnegative sums on the right are finite, being bounded by $t$.
Dropping the term $-t h_t^S(x)$ and conditioning on the exit event gives
\eqref{eq:exit-density-bound}.
\end{proof}

The identity \eqref{eq:exit-density-identity} and the first inequality in
\eqref{eq:exit-density-bound} use only independence of the rate-one
Poisson clock and the embedded chain.  They therefore hold for every
rate-one kernel $Q$.  If an exit event has probability zero, we may use the
unconditional expressions rather than the conditional expectation.

\Needspace{5\baselineskip}
We can now combine the Russo decompositions with the two jump-count
estimates.

\begin{proposition}
\label{prop:reverse-differential-inequality}
For every finite $\Lambda\ni o$ and all $\lambda,t>0$,

\begin{equation}
 \partial_t f_\Lambda(\lambda,t)
 \le
 \frac{\lambda C(\Delta,\lambda,t)}{t}\,
 \partial_\lambda f_\Lambda(\lambda,t).
 \label{eq:reverse-differential-inequality}
\end{equation}

Together, \eqref{eq:forward-differential-comparison} and
\eqref{eq:reverse-differential-inequality} give

\begin{equation}
 \frac{\lambda}{e^t-1}\,\partial_\lambda f_\Lambda
 \le \partial_t f_\Lambda
 \le \frac{\lambda C(\Delta,\lambda,t)}{t}\,
       \partial_\lambda f_\Lambda.
 \label{eq:two-sided-differential-inequality}
\end{equation}
\end{proposition}

\begin{proof}
Substitute Lemma~\ref{lem:exit-density} into \eqref{eq:russo-time} and
apply Lemma~\ref{lem:exit-weighted-jumps} for each value of $S$ to obtain

\begin{align*}
 \partial_t f_\Lambda(\lambda,t)
 &\le \frac{\lambda}{t}
 \sum_{\substack{S\subseteq\Lambda\\o\in S}}
 \P_{\lambda,t}(L=S)
 \sum_{x\in S}a_{\lambda,t}^S(x)h_t^S(x)
 \E_x[N_t\mid\tau_{S^c}\le t]\\
 &\le \frac{\lambda C(\Delta,\lambda,t)}{t}
 \sum_{\substack{S\subseteq\Lambda\\o\in S}}
 \P_{\lambda,t}(L=S)
 \sum_{x\in S}a_{\lambda,t}^S(x)h_t^S(x)\\
 &=\frac{\lambda C(\Delta,\lambda,t)}{t}
   \partial_\lambda f_\Lambda(\lambda,t),
\end{align*}

where the last equality is \eqref{eq:russo-lambda}.  The lower inequality in
\eqref{eq:two-sided-differential-inequality} is
\eqref{eq:forward-differential-comparison}.
\end{proof}

We close the section by turning the resulting differential inequality into a parameter comparison by
integrating along suitable curves in the $(\lambda,t)$-plane.

\begin{theorem}
\label{thm:reverse-characteristic-comparison}
Let

\[
 \cK=[a,b]\times[u,v]\Subset(0,\infty)^2,
 \qquad
 M_{\cK}:=\max_{(\lambda,t)\in\cK}
 \frac{C(\Delta,\lambda,t)}{t}.
\]

Suppose $u\le s<t\le v$, $\lambda_0\in[a,b]$, and

\[
 \lambda_r:=\lambda_0e^{-M_{\cK}(r-s)}\in[a,b]
 \qquad(s\le r\le t).
\]

Then

\begin{equation}
 \theta(\lambda_0e^{-M_{\cK}(t-s)},t)
 \le \theta(\lambda_0,s).
 \label{eq:reverse-characteristic-comparison}
\end{equation}
\end{theorem}

\begin{proof}
For every finite $\Lambda\ni o$, the chain rule,
\eqref{eq:reverse-differential-inequality}, and monotonicity in $\lambda$
give

\begin{align*}
 \frac{\mathrm d}{\mathrm d r}f_\Lambda(\lambda_r,r)
 &=\partial_t f_\Lambda(\lambda_r,r)
   -M_{\cK}\lambda_r\partial_\lambda f_\Lambda(\lambda_r,r)\\
 &\le\lambda_r
 \left(\frac{C(\Delta,\lambda_r,r)}{r}-M_{\cK}\right)
 \partial_\lambda f_\Lambda(\lambda_r,r)
 \le0.
\end{align*}

Integrate from $s$ to $t$.  For $\Lambda_n=B_G(o,n)$, the proof of
Corollary~\ref{cor:infinite-volume-comparison} identifies the decreasing
limit of $\{o\rightsquigarrow_{\Lambda_n}\Lambda_n^c\}$ with
$\{|\cC_o|=\infty\}$.
Continuity from above gives \eqref{eq:reverse-characteristic-comparison}.
\end{proof}

\section{Critical curves and Conjecture~1}
\label{sec:critical-curves}

Throughout this section $G$ is an infinite, connected, locally finite,
transitive graph of superlinear growth, and the frogs follow rate-one
continuous-time simple random walk.  We abbreviate

\[
 \theta(\lambda,t):=\theta_Q(\lambda,t),
 \qquad
 f(t):=\lambda_c(t).
\]

By \cite[Theorem~1.4(1),(2)]{AngelDeLaRivaHermonShi2026} and
\cite{DuminilCopinGoswamiRaoufiSeveroYadin2020},

\begin{equation}
 0<\lambda_c(t)<\infty,\qquad t>0,
 \label{eq:critical-density-positive-finite}
\end{equation}

\cite[Proposition~2.1]{AngelDeLaRivaHermonShi2026} also gives

\begin{equation}
 \lambda_c(t)\ge\frac1t.
 \label{eq:critical-density-small-time}
\end{equation}

Monotonicity in $\lambda$ and the definition of $\lambda_c(t)$ imply

\begin{equation}
 \lambda<\lambda_c(t)\Longrightarrow\theta(\lambda,t)=0,
 \qquad
 \lambda>\lambda_c(t)\Longrightarrow\theta(\lambda,t)>0.
 \label{eq:critical-zero-supercritical}
\end{equation}

Indeed, the zero set is downward closed and contains $0$, so
\eqref{eq:critical-zero-supercritical} follows from its supremum.

\subsection{Regularity of the finite phase boundary}

We now apply the forward and reverse parameter comparisons to the
critical-density curve.

\begin{theorem}
\label{thm:continuous-critical-density}
The map

\[
 t\longmapsto\lambda_c(t)
\]

is continuous and strictly decreasing on $(0,\infty)$.  Moreover,
$-\log\lambda_c$ is locally bi-Lipschitz.
\end{theorem}

\begin{proof}
Strict decrease follows from
\eqref{eq:critical-density-positive-finite} and
Theorem~\ref{thm:strict-critical-density}.  It remains to prove continuity.
Fix $t_0>0$.

For right-continuity, take $0<\lambda<f(t_0)$.  By
\eqref{eq:critical-zero-supercritical}, $\theta(\lambda,t_0)=0$.  Choose
$\eta>0$ and a compact rectangle

\[
 \cK=[\lambda/2,\lambda]\times[t_0,t_0+\eta]
 \Subset(0,\infty)^2,
\]

and let $M_{\cK}$ be as in
Theorem~\ref{thm:reverse-characteristic-comparison}.  For all sufficiently
small $\delta>0$, the characteristic stays in $\cK$, and
\eqref{eq:reverse-characteristic-comparison} gives

\[
 \theta(\lambda e^{-M_{\cK}\delta},t_0+\delta)=0.
\]

Hence

\[
 f(t_0+\delta)\ge\lambda e^{-M_{\cK}\delta}.
\]

Let $\delta\downarrow0$ and then $\lambda\uparrow f(t_0)$.  Since $f$
is nonincreasing, this proves right-continuity.

For left-continuity, put

\[
 L:=\lim_{s\uparrow t_0}f(s).
\]

This limit is finite because $f(s)\le f(t_0/2)<\infty$ for
$s\in[t_0/2,t_0)$.  Take $0<\lambda<L$.  Equation
\eqref{eq:critical-zero-supercritical} gives $\theta(\lambda,s)=0$ for all
$s<t_0$ sufficiently close to $t_0$.
Choose $\eta\in(0,t_0)$, form the compact rectangle

\[
 \cK=[\lambda/2,\lambda]\times[t_0-\eta,t_0],
\]

and apply Theorem~\ref{thm:reverse-characteristic-comparison} from $s$ to
$t_0$.  For $s$ sufficiently close to $t_0$,

\[
 f(t_0)\ge\lambda e^{-M_{\cK}(t_0-s)}.
\]

Let $s\uparrow t_0$ and then $\lambda\uparrow L$.  Thus
$f(t_0)\ge L$, while monotonicity gives $f(t_0)\le L$.  Hence $f$ is
left-continuous.

Equations~\eqref{eq:critical-density-comparison} and
\eqref{eq:reverse-characteristic-comparison} also give local quantitative
control.  Fix $t_0>0$ and put $f_0=f(t_0)$.  By continuity, choose a compact
time interval $I_0\Subset(0,\infty)$ about $t_0$ such that

\[
 \frac{3f_0}{4}\le f(r)\le\frac{5f_0}{4},
 \qquad r\in I_0.
\]

Set
$\cK=[f_0/4,3f_0/2]\times I_0$ and form $M_{\cK}$.  Shrink to an interval
$I\subset I_0$ about $t_0$ whose diameter is at most
$(\log 2)/M_{\cK}$.  If $s<t$ lie in $I$, take
$f_0/2<\lambda<f(s)$.  The characteristic
$r\mapsto\lambda e^{-M_{\cK}(r-s)}$ satisfies

\[
 \lambda e^{-M_{\cK}(r-s)}\in[f_0/4,3f_0/2],
 \qquad s\le r\le t,
\]

so the characteristic stays in $\cK$.  Moreover,
\eqref{eq:critical-zero-supercritical} gives $\theta(\lambda,s)=0$, and
Theorem~\ref{thm:reverse-characteristic-comparison} gives

\[
 f(t)\ge \lambda e^{-M_{\cK}(t-s)}.
\]

Letting $\lambda\uparrow f(s)$ and using the upper bound
\eqref{eq:critical-density-comparison} yields

\begin{equation}
 e^{-M_{\cK}(t-s)}\lambda_c(s)
 \le\lambda_c(t)
 \le\frac{1-e^{-s}}{1-e^{-t}}\lambda_c(s).
 \label{eq:critical-curve-two-sided}
\end{equation}

Taking logarithms in \eqref{eq:critical-curve-two-sided} gives
\[
 \int_s^t\frac{\di r}{e^r-1}
 =\log\frac{1-e^{-t}}{1-e^{-s}}
 \le\log\frac{\lambda_c(s)}{\lambda_c(t)}
 \le M_{\cK}(t-s).
\]
The integrand is bounded away from zero on $I$.  Hence
$-\log\lambda_c$ is bi-Lipschitz on a neighbourhood of $t_0$, and therefore
locally bi-Lipschitz on $(0,\infty)$.
\end{proof}

Next we transfer this regularity to the critical-lifespan curve using
generalised inversion.

Recall the critical lifespan

\[
 t_c(\lambda)
 :=\sup\{t\ge0:\theta(\lambda,t)=0\},
\]

and the definition

\[
 \lambda_\infty:=\lim_{t\to\infty}\lambda_c(t)\in[0,\infty).
\]

\begin{proposition}
\label{prop:inverse-critical-continuity}
The map $t_c$ is finite, continuous, strictly decreasing, and locally
bi-Lipschitz on
$(\lambda_\infty,\infty)$, while

\[
 t_c(\lambda)=\infty,
 \qquad 0<\lambda\le\lambda_\infty.
\]

In particular, if $t_c(\lambda)<\infty$ for every $\lambda>0$, then
$\lambda\mapsto t_c(\lambda)$ is continuous, strictly decreasing, and locally
bi-Lipschitz on
$(0,\infty)$.
\end{proposition}

\begin{proof}
For fixed $\lambda>0$, \eqref{eq:critical-zero-supercritical} gives

\begin{equation}
 \{t>0:f(t)>\lambda\}
 \subseteq
 \{t>0:\theta(\lambda,t)=0\}
 \subseteq
 \{t>0:f(t)\ge\lambda\}.
 \label{eq:critical-inverse-sandwich}
\end{equation}

Since $f$ is strictly decreasing, the difference between the first and third
sets in \eqref{eq:critical-inverse-sandwich}
contains at most the unique time, if any, at which $f(t)=\lambda$.
Moreover, $\{t>0:f(t)>\lambda\}$ is nonempty by
\eqref{eq:critical-density-small-time}.  The three sets therefore have the
same supremum, and

\begin{equation}
 t_c(\lambda)=\sup\{t>0:f(t)>\lambda\}.
 \label{eq:generalized-inverse}
\end{equation}

By \eqref{eq:critical-density-small-time}, $f(t)\to\infty$ as
$t\downarrow0$.  Theorem~\ref{thm:continuous-critical-density} therefore
implies that, for every $\lambda>\lambda_\infty$, there is a unique
$t_\lambda\in(0,\infty)$ such that $f(t_\lambda)=\lambda$.  Equation
\eqref{eq:generalized-inverse} gives $t_c(\lambda)=t_\lambda$.  The inverse
of a continuous strictly decreasing function is continuous and strictly
decreasing.

To obtain the quantitative assertion, fix
$K=[a,b]\Subset(\lambda_\infty,\infty)$ and put
\[
 I:=[t_c(b),t_c(a)]\Subset(0,\infty).
\]
Compactness and a finite chaining of the local estimates in
Theorem~\ref{thm:continuous-critical-density} give constants
$0<c_I\le C_I<\infty$ such that, whenever
$a\le\lambda_1<\lambda_2\le b$ and $t_i:=t_c(\lambda_i)$,
\[
 c_I(t_1-t_2)
 \le\log\frac{\lambda_2}{\lambda_1}
 \le C_I(t_1-t_2).
\]
Indeed, we may cover $I$ by finitely many open intervals on which the two-sided
local estimate holds.  By compactness, choose $\delta>0$ such that every
subinterval of $I$ of length less than $\delta$ is contained in one of these
intervals.
Partition $[t_2,t_1]$ into pieces of length below $\delta$ and sum the
logarithmic estimates over the pieces.
Since
\[
 \frac{\lambda_2-\lambda_1}{b}
 \le\log\frac{\lambda_2}{\lambda_1}
 \le\frac{\lambda_2-\lambda_1}{a},
\]
it follows that
\[
 \frac{\lambda_2-\lambda_1}{bC_I}
 \le t_c(\lambda_1)-t_c(\lambda_2)
 \le\frac{\lambda_2-\lambda_1}{ac_I}.
\]
Thus $t_c$ is bi-Lipschitz on every compact subinterval of
$(\lambda_\infty,\infty)$.

If $0<\lambda\le\lambda_\infty$, strict decrease of $f$ and the definition
of its limit imply $\lambda<f(t)$ for every finite $t$.  Equation
\eqref{eq:generalized-inverse} then gives $t_c(\lambda)=\infty$.
\end{proof}

\subsection{The exact short-lifespan constant}

\begin{proposition}
\label{prop:small-lifespan-limit}
Let $\Delta$ be the degree of $G$, and let $p_c(G)$ be its
bond-percolation threshold.  Then
\begin{equation}
 L_G
 :=\lim_{t\downarrow0}(1-e^{-t})\lambda_c(t)
 =\lim_{t\downarrow0}t\lambda_c(t)
 =\Delta\log\frac1{1-p_c(G)}.
 \label{eq:small-exact-limit}
\end{equation}
For every $t>0$, the corresponding finite-time upper bound is
\begin{equation}
 (1-e^{-t})\lambda_c(t)\le L_G.
 \label{eq:small-exact-upper}
\end{equation}
Moreover,
\[
 \lim_{\lambda\to\infty}\lambda t_c(\lambda)
 =\Delta\log\frac1{1-p_c(G)}.
\]
\end{proposition}

\begin{proof}
Put
\[
 c_*:=\Delta\log\frac1{1-p_c(G)},
 \qquad
 p(c):=1-e^{-c/\Delta}.
\]
The superlinear-growth assumption gives $p_c(G)<1$, and hence
$c_*<\infty$.

\emph{Directed first-jump percolation.}
We first recall that independent directed nearest-neighbour edges, each
present with probability $p$, have a root out-cluster with the same law as
the root cluster of Bernoulli bond percolation at parameter $p$.  To see
this, explore from the root by querying edges from reached to unreached
vertices in a deterministic breadth-first order.  At most one orientation
of each unordered edge is ever queried.  If its first query succeeds,
both endpoints become reached.  If it fails and the other endpoint is
later reached by a different route, the reverse orientation leads to an
already reached vertex and need not be queried.  Each new query therefore
uses an independent Bernoulli variable, just as in ordinary bond
exploration.  Exhaustion extends the identity to the infinite graph.
This is the exploration argument of
\cite[Lemma~3.5]{BaeumlerJahnelKoepplLodewijksReevesTobias2026},
which applies without change to the present graph.

Fix $c>c_*$ and $t>0$, and take
$\lambda=c/(1-e^{-t})$.  For every oriented nearest-neighbour edge
$x\to y$, count the frogs born at $x$ whose first jump occurs by time
$t$ and lands at $y$.  Poisson marking makes these counts independent
over all oriented edges, each with mean
\[
 \frac{\lambda(1-e^{-t})}{\Delta}=\frac c\Delta.
\]
Their positivity indicators are therefore independent directed edges
with parameter $p(c)>p_c(G)$.  Each such edge belongs to the frog
activation relation.  The preceding cluster-law identity gives survival
with positive probability, and consequently
$(1-e^{-t})\lambda_c(t)\le c$.  Letting $c\downarrow c_*$ gives
the upper bound in \eqref{eq:small-exact-upper} with $c_*$ on the right.
It remains to prove that the short-lifespan limit equals $c_*$.

\emph{A subcritical decomposition.}
Now fix $0<c<c_*$ and again take
$\lambda=c/(1-e^{-t})$.  Let $N\sim\Poi(t)$ denote the jump count
of a single frog.  At every source, split the marked Poisson population
into frogs with $N=1$ and frogs with $N\ge2$.  These populations are
independent of one another, and the populations at distinct sources are
independent.  Frogs with $N=0$ have no effect on activation.

The $N=1$ population generates independent directed nearest-neighbour
edges with parameter
\[
 p_1(t)
 =1-\exp\left\{-\frac{\lambda te^{-t}}{\Delta}\right\}
 =1-\exp\left\{-\frac{ct}{\Delta(e^t-1)}\right\}
 \le p(c)<p_c(G).
\]
Let $\tau_t(x,y)$ be its directed connection probability, with
$\tau_t(x,x)=1$.  Write $\E_p^{\mathrm{bond}}$ for expectation under
Bernoulli bond percolation of parameter $p$, and let
$\cC_o^{\mathrm{bond}}$ be its root cluster.  The directed/bond
cluster-law identity, transitivity, and monotonicity give
\begin{equation}
 \sup_{x\in V}\sum_{y\in V}\tau_t(x,y)\le\chi,
 \qquad
 \chi:=\E_{p(c)}^{\mathrm{bond}}
       [|\cC_o^{\mathrm{bond}}|]<\infty.
 \label{eq:small-exact-base-susceptibility}
\end{equation}
The finiteness assertion is subcritical Bernoulli sharpness on
transitive graphs
\cite{DuminilCopinTassion2016}.

For $y\ne x$, put
\[
 K_t(x,y)
 :=\lambda\P_x\bigl(N\ge2,\
       y\in\{X_s:0\le s\le t\}\bigr),
 \qquad K_t(x,x):=0,
\]
where $\P_x$ is the law of an independent test walk.  Poisson marking
shows that the probability that the $N\ge2$ population supplies a
macro-edge $x\to y$ is $1-e^{-K_t(x,y)}\le K_t(x,y)$.  Such
events can be dependent when they share a source; we use independence
only between distinct sources.  Since a walk visits at most $N$
vertices other than its starting point,
\begin{equation}
 \sup_{x\in V}\sum_{y\in V}K_t(x,y)
 \le\lambda\E[N\1_{\{N\ge2\}}]
 =\lambda t(1-e^{-t})
 =ct.
 \label{eq:small-exact-extra-row-sum}
\end{equation}

\emph{A connection expansion.}
Every connection from $x$ to $y$ in the full activation relation has
a finite simple directed witnessing path.  Label each of its edges
by one population that supplies it.  If the path uses $n\ge1$
macro-edges from the $N\ge2$ population, list them in order as
$u_i\to v_i$, $1\le i\le n$.  They are joined by paths in the
$N=1$ edge field, from $x$ to $u_1$, from $v_i$ to $u_{i+1}$,
and from $v_n$ to $y$.  Some joining paths may have length zero.
The case $n=0$ is a connection in the $N=1$ field alone.

Because the full witnessing path is simple, the sources
$u_1,\ldots,u_n$ are distinct and the joining paths have disjoint
directed-edge witnesses.  Fix an endpoint tuple with distinct
sources.  Independence across these sources bounds the probability
of the required $N\ge2$ macro-edges by
$\prod_{i=1}^nK_t(u_i,v_i)$.  This population is independent of the
$N=1$ field.  In the latter field, the van den Berg--Kesten
inequality for increasing events
\cite{VanDenBergKesten1985} bounds the probability of the
disjoint joining connections by
\[
 \tau_t(x,u_1)
 \left(\prod_{i=1}^{n-1}\tau_t(v_i,u_{i+1})\right)
 \tau_t(v_n,y).
\]
Here the independent Bernoulli coordinates are the oriented edges.
The inequality can first be applied with all joining paths restricted
to a finite vertex set, and then extended by exhaustion, since their
disjoint witnesses are finite.

Take a union bound over $n$ and the endpoint tuples with distinct
sources.  After applying the preceding probability bounds, remove
the distinctness restriction from the sum of nonnegative products.
Thus, for $n\ge1$, the contribution is at most
\[
 \sum_{u_1,v_1,\ldots,u_n,v_n\in V}
 \tau_t(x,u_1)
 \left(\prod_{i=1}^nK_t(u_i,v_i)\right)
 \left(\prod_{i=1}^{n-1}\tau_t(v_i,u_{i+1})\right)
 \tau_t(v_n,y).
\]
For $n=0$ the bound is $\tau_t(x,y)$.  Summing over $y$ and
successively using
\eqref{eq:small-exact-base-susceptibility} and
\eqref{eq:small-exact-extra-row-sum} now gives, by Tonelli's theorem,
\begin{equation}
 \E_{\lambda,t}[|\cC_x|]
 \le\sum_{n\ge0}\chi^{n+1}(ct)^n
 =\frac{\chi}{1-ct\chi},
 \qquad ct\chi<1.
 \label{eq:small-exact-full-susceptibility}
\end{equation}
In particular, the frog model becomes extinct at density
$c/(1-e^{-t})$ for all sufficiently small $t$.  Hence
\[
 \liminf_{t\downarrow0}(1-e^{-t})\lambda_c(t)\ge c.
\]
Let $c\uparrow c_*$ and use \eqref{eq:small-exact-upper}.
This proves convergence to $c_*$.  Since
$t/(1-e^{-t})\to1$, the two limits in
\eqref{eq:small-exact-limit} coincide.

Finally, Proposition~\ref{prop:inverse-critical-continuity} gives
$\lambda_c(t_c(\lambda))=\lambda$ for
$\lambda>\lambda_\infty$.  Moreover,
$t_c(\lambda)\downarrow0$ as $\lambda\to\infty$; otherwise
monotonicity of $\lambda_c$ and its finiteness at every positive
time would bound $\lambda_c(t_c(\lambda))$ uniformly.  Therefore
\[
 \lambda t_c(\lambda)
 =t_c(\lambda)\lambda_c(t_c(\lambda))
 \longrightarrow c_*.
\]
\end{proof}

\subsection{The endpoint assertion}

It remains to identify when the inverse curve is finite on all of
$(0,\infty)$, which determines the exact status of Conjecture~1.

\begin{corollary}
\label{cor:conjecture-1}
Let $G$ be a transitive graph of superlinear growth.  If $G$ is
nonamenable, or if $G$ has polynomial growth, then both critical curves in
Conjecture~1 are continuous and strictly decreasing
on $(0,\infty)$.
\end{corollary}

\begin{proof}
For nonamenable graphs, finiteness of $t_c(\lambda)$ for every
$\lambda>0$ is \cite[Theorem~1.4(3)]{AngelDeLaRivaHermonShi2026}.  For
transitive graphs of superlinear polynomial growth it is
\cite[Theorem~1.5(2)]{AngelDeLaRivaHermonShi2026}.  Thus
$\lambda_\infty=0$ in either case.  Apply
Theorem~\ref{thm:continuous-critical-density} and
Proposition~\ref{prop:inverse-critical-continuity}.
\end{proof}

On an arbitrary transitive graph of superlinear growth,
Theorem~\ref{thm:continuous-critical-density} applies without an endpoint
assumption.  Proposition~\ref{prop:inverse-critical-continuity} shows that
the full assertion for $t_c$ is equivalent to

\begin{equation}
 \lambda_\infty=0.
 \label{eq:lambda-infinity-boundary}
\end{equation}

Condition~\eqref{eq:lambda-infinity-boundary} is the finiteness assertion
$t_c(\lambda)<\infty$ in
\cite[Conjecture~1.1]{AngelDeLaRivaHermonShi2026}.  The comparison arguments above separate this endpoint assertion from
regularity of the finite phase boundary.

A reverse counterpart to \eqref{eq:poisson-clock-contraction} with a
constant uniform in $n$ is impossible.  For $0<s<t$,

\[
 \frac{\P(N_t\ge n)}{\P(N_s\ge n)}
 \sim e^{-(t-s)}\left(\frac ts\right)^n
 \longrightarrow\infty,
\]

because $\P(N_u\ge n)\sim\P(N_u=n)$ for fixed $u>0$.  The pointwise two-step
estimate in Lemma~\ref{lem:bridge-exit-bound} provides a different mechanism
for the kernels considered in Section~\ref{sec:bridge-kernels}.

\section{The two-step condition and long-range extensions}
\label{sec:bridge-kernels}

Theorem~\ref{thm:local-to-global} and
Proposition~\ref{prop:density-lifespan-tradeoff} hold for every rate-one
kernel $Q$.  Proposition~\ref{prop:russo-all-kernels} extends the Russo
formulas to such kernels, and Lemma~\ref{lem:bridge-exit-bound} bounds the
exit-weighted jump count under \eqref{eq:bridge-condition}.

We begin by verifying that the Russo decomposition itself does not use
nearest-neighbour geometry.

Recall, for finite $\Lambda\ni o$, the exit probability

\begin{equation}
 f_\Lambda^Q(\lambda,t)
 :=\P_{\lambda,t,Q}
 \bigl(o\rightsquigarrow_\Lambda\Lambda^c\bigr).
 \label{eq:finite-volume-exit-Q}
\end{equation}

Define $L_\Lambda^Q$, $a_{\lambda,t,Q}^S$, $h_{t,Q}^S$, and $q_{t,Q}^S$
by \eqref{eq:nonexit-set}, \eqref{eq:internal-activation-weight}, and
\eqref{eq:exit-probability-density}, with the simple random walk replaced by
the rate-one $Q$-walk.

\begin{proposition}
\label{prop:russo-all-kernels}
For every Markov kernel $Q$ on $V$ and every $\lambda,t>0$,

\begin{align}
 \partial_\lambda f_\Lambda^Q(\lambda,t)
 &=\sum_{\substack{S\subseteq\Lambda\\o\in S}}
   \P_{\lambda,t,Q}(L_\Lambda^Q=S)
   \sum_{x\in S}a_{\lambda,t,Q}^S(x)h_{t,Q}^S(x),
 \label{eq:russo-lambda-Q}\\
 \partial_t f_\Lambda^Q(\lambda,t)
 &=\lambda\sum_{\substack{S\subseteq\Lambda\\o\in S}}
   \P_{\lambda,t,Q}(L_\Lambda^Q=S)
   \sum_{x\in S}a_{\lambda,t,Q}^S(x)q_{t,Q}^S(x).
 \label{eq:russo-time-Q}
\end{align}
\end{proposition}

\begin{proof}
Identify all of $\Lambda^c$ with one symbol $\partial$.  For a one-particle
$Q$-walk started at $x\in\Lambda$, define its time-$t$ trace stopped and
collapsed at the first exit from $\Lambda$ by
\[
 T_{x,t}^{\Lambda,Q}:=
 \begin{cases}
  \{X_s:0\le s\le t\},&\tau_{\Lambda^c}>t,\\
  \{X_s:0\le s<\tau_{\Lambda^c}\}\cup\{\partial\},
    &\tau_{\Lambda^c}\le t.
 \end{cases}
\]
For $B\subseteq\Lambda\cup\{\partial\}$, let $Y_{x,B}$ count the frogs at
$x$ whose collapsed trace equals $B$.  Poisson thinning makes the finite
family $(Y_{x,B})$ independent, with
\[
 Y_{x,B}\sim
 \Poi\!\left(\lambda\P_x^Q(T_{x,t}^{\Lambda,Q}=B)\right).
\]

For $S\subseteq\Lambda$ with $o\in S$, let $\mathsf{Stay}(S)$ be the event
that every time-$t$ trace born in $S$ stays in $S$, and let
\[
 \mathsf{Exit}(S)
 :=\bigcap_{y\in\Lambda\setminus S}
   \{y\text{ reaches }\Lambda^c
     \text{ using only clouds based in }\Lambda\setminus S\}.
\]
Then
\begin{equation}
 \{L_\Lambda^Q=S\}
 =\mathsf{Stay}(S)\cap\mathsf{Exit}(S).
 \label{eq:nonexit-factorization-Q}
\end{equation}
Indeed, on $\{L_\Lambda^Q=S\}$ a trace born in $S$ cannot hit
$\Lambda^c$, nor can it hit $z\in\Lambda\setminus S$, since such a $z$
reaches $\Lambda^c$.  Any path witnessing that a vertex of
$\Lambda\setminus S$ reaches $\Lambda^c$ must avoid $S$, since otherwise
its suffix would make a vertex of $S$ reach $\Lambda^c$.  This proves one
inclusion in \eqref{eq:nonexit-factorization-Q}.  Conversely, under
$\mathsf{Stay}(S)$ no path from $S$ can leave $S$, while
$\mathsf{Exit}(S)$ makes every vertex of $\Lambda\setminus S$ reach
$\Lambda^c$.

At sources in $S$, marked Poisson thinning separates traces that stay in
$S$ from traces that leave $S$.  Hence the internal activation event
defining $a_{\lambda,t,Q}^S(x)$ depends only on staying traces,
$\mathsf{Stay}(S)$ depends only on exiting-trace counts, and
$\mathsf{Exit}(S)$ depends only on clouds based in $\Lambda\setminus S$.
These three objects are independent.  The Poisson Russo formula now
differentiates \eqref{eq:finite-volume-exit-Q} by adding one trace of type
$(x,B)$.  Conditional on $L_\Lambda^Q=S$, it is pivotal exactly when its
source $x\in S$ is activated internally and the added trace exits $S$.
Summing over $B$ gives \eqref{eq:russo-lambda-Q}.  This is the argument of
\cite[Equations~(3.3)--(3.6)]{AngelDeLaRivaHermonShi2026}; no property of the
embedded chain was used.

For the time derivative, give all frogs born at $x$ a lifespan $t_x$, write
$\mathbf t=(t_x)_{x\in\Lambda}$, and denote the corresponding finite-volume
probability by $F_\Lambda^Q(\mathbf t)$.  On
$\{L_\Lambda^Q=S\}$, increasing $t_x$ from $t$ to
$t+\varepsilon$ is pivotal exactly when $x$ is internally activated and at
least one frog born at $x$ first exits $S$ during
$(t,t+\varepsilon]$.  A simple activation path ending at $x$ uses no frog
born at $x$.  Marking the Poisson cloud at $x$ according to whether a frog
exits by time $t$, first exits during $(t,t+\varepsilon]$, or stays through
$t+\varepsilon$, and using \eqref{eq:nonexit-factorization-Q}, gives the
independent factor
\[
 1-\exp\{-\lambda\P_x^Q(\tau_{S^c}\in(t,t+\varepsilon])\}.
\]

More explicitly, let $\mathbf t=(t,\ldots,t)$ and let $\mathbf e_x$ be
the coordinate vector at $x$.  The finite increment is
\begin{equation}
\begin{aligned}
 &F_\Lambda^Q(\mathbf t+\varepsilon\mathbf e_x)
       -F_\Lambda^Q(\mathbf t)\\
 &\quad=
 \sum_{\substack{S\subseteq\Lambda\\o,x\in S}}
 \P_{\lambda,t,Q}(L_\Lambda^Q=S)
 a_{\lambda,t,Q}^S(x)
 \left(1-\exp\left\{-\lambda\P_x^Q
       \bigl(\tau_{S^c}\in(t,t+\varepsilon]\bigr)\right\}\right).
\end{aligned}
\label{eq:russo-time-finite-increment-Q}
\end{equation}
All configuration probabilities on the right are evaluated at the
uniform lifespan $t$.

After division by $\varepsilon$ and letting $\varepsilon\downarrow0$, this
gives
\[
 \partial_{t_x}F_\Lambda^Q(t,\ldots,t)
 =
 \lambda\sum_{\substack{S\subseteq\Lambda\\o,x\in S}}
 \P_{\lambda,t,Q}(L_\Lambda^Q=S)
 a_{\lambda,t,Q}^S(x)q_{t,Q}^S(x).
\]
The first-exit distribution is a mixture of
$\operatorname{Gamma}(k,1)$ laws because the rate-one clock is independent
of the embedded $Q$-chain.

To justify the diagonal chain rule, fix
$J=[a,b]\Subset(0,\infty)$.  Before absorption at $\partial$, the pair
consisting of the current vertex and the visited subset of $\Lambda$ is a
finite-state continuous-time Markov chain.  Thus every collapsed-trace
probability is continuously differentiable in time.  The Gamma-series
domination \eqref{eq:gamma-mixture-domination} makes the first-exit
derivatives locally uniform on $J$.  Finite-volume reachability depends only
on whether each of the finitely many variables $Y_{x,B}$ vanishes, so
$F_\Lambda^Q$ is $C^1$ on $J^\Lambda$.
Therefore
\[
 \frac{\mathrm d}{\mathrm dt}F_\Lambda^Q(t,\ldots,t)
 =\sum_{x\in\Lambda}
   \partial_{t_x}F_\Lambda^Q(t,\ldots,t),
\]
and summing the coordinate derivatives gives
\eqref{eq:russo-time-Q}.
\end{proof}

Next we prove the one-big-jump estimate that replaces the simple-random-walk
weighted bound.

Assume that, for some $\kappa\in[0,\infty)$,

\begin{equation}
 Q^2(x,y)\le\kappa Q(x,y),
 \qquad x\ne y.
 \label{eq:bridge-condition}
\end{equation}

We call \eqref{eq:bridge-condition} the \emph{two-step condition}.
It is an off-diagonal sub-convolution bound; in particular, a two-step route
to a different vertex cannot be much more likely than one direct jump.

\begin{lemma}
\label{lem:bridge-exit-bound}
Suppose \eqref{eq:bridge-condition} holds and put $r:=1+\kappa$.  For every
$t>0$, $x\in V$, and $A\subseteq V\setminus\{x\}$,

\begin{equation}
 \E_x^Q\bigl[N_t\1_{\{\tau_A\le t\}}\bigr]
 \le B_\kappa(t)\P_x^Q(\tau_A\le t),
 \qquad
 B_\kappa(t):=
 \frac{t(1+rt)e^{(r-1)t}}{1-e^{-t}}.
 \label{eq:bridge-exit-bound}
\end{equation}

The function $B_\kappa$ is finite and continuous on $(0,\infty)$ and extends
continuously to $B_\kappa(0)=1$.
\end{lemma}

\begin{proof}
Induction gives, for every $n\ge1$ and $x\ne y$,

\begin{equation}
 Q^n(x,y)\le r^{n-1}Q(x,y).
 \label{eq:multistep-bridge}
\end{equation}

The case $n=1$ is immediate.  If \eqref{eq:multistep-bridge} holds at $n$,
then

\begin{align*}
 Q^{n+1}(x,y)
 &=Q^n(x,x)Q(x,y)+\sum_{z\ne x}Q^n(x,z)Q(z,y)\\
 &\le Q(x,y)+r^{n-1}Q^2(x,y)\\
 &\le(1+\kappa r^{n-1})Q(x,y)
 \le r^nQ(x,y),
\end{align*}

which proves \eqref{eq:multistep-bridge}.

Let $(Y_n)_{n\ge0}$ be the embedded chain, independent of $N_t$, and set
$H_A:=\inf\{n\ge1:Y_n\in A\}$.  For a deterministic jump budget $n$,
\eqref{eq:multistep-bridge} and a union bound give

\[
 \P_x^Q(H_A\le n)
 \le\sum_{j=1}^nQ^j(x,A)
 \le Q(x,A)\sum_{j=1}^nr^{j-1}
 \le Q(x,A)nr^{n-1}.
\]

Consequently,

\begin{equation}
 \E_x^Q\bigl[N_t\1_{\{H_A\le N_t\}}\bigr]
 \le Q(x,A)\E[N_t^2r^{N_t-1}].
 \label{eq:bridge-numerator}
\end{equation}

On the other hand, exiting on the first jump yields

\begin{equation}
 \P_x^Q(H_A\le N_t)\ge(1-e^{-t})Q(x,A).
 \label{eq:bridge-denominator}
\end{equation}

For $F(r):=\E[r^{N_t}]=e^{t(r-1)}$,

\begin{equation}
 \E[N_t^2r^{N_t-1}]=rF''(r)+F'(r)
 =t(1+tr)e^{t(r-1)}.
 \label{eq:poisson-bridge-moment}
\end{equation}

Equations~\eqref{eq:bridge-numerator}--\eqref{eq:poisson-bridge-moment} and
$\{\tau_A\le t\}=\{H_A\le N_t\}$ prove
\eqref{eq:bridge-exit-bound}.  If $Q(x,A)=0$, then
\eqref{eq:multistep-bridge} makes both sides of
\eqref{eq:bridge-exit-bound} zero; no division by $Q(x,A)$ is required.
\end{proof}

With this estimate in hand, the reverse-comparison argument extends to
kernels satisfying the two-step condition.

For a kernel $Q$ equivariant under a transitive subgroup
$\Gamma\le\Aut(G)$, define

\[
 t_c^Q(\lambda):=\sup\{t\ge0:\theta_Q(\lambda,t)=0\},
 \qquad
 \lambda_\infty^Q:=\lim_{t\to\infty}\lambda_c^Q(t).
\]

\begin{theorem}
\label{thm:bridge-kernel-regularity}
Let $G$ be an infinite, connected, locally finite, transitive graph of
superlinear growth.  Let $Q$ be an irreducible Markov kernel on $V(G)$,
equivariant under a transitive subgroup $\Gamma\le\Aut(G)$, and satisfying
\eqref{eq:bridge-condition}.  Then $\lambda_c^Q$ is continuous and strictly
decreasing on $(0,\infty)$, and $-\log\lambda_c^Q$ is locally bi-Lipschitz.

Moreover, $t_c^Q$ is finite, continuous, strictly decreasing, and locally
bi-Lipschitz on
$(\lambda_\infty^Q,\infty)$, while

\[
 t_c^Q(\lambda)=\infty,
 \qquad 0<\lambda\le\lambda_\infty^Q.
\]
\end{theorem}

\begin{proof}
The branching comparison of
\cite[Proposition~2.1]{AngelDeLaRivaHermonShi2026} uses only
that the number of new vertices exposed by a collection of trajectories is
at most its total number of jumps.  It therefore applies to $Q$ and gives

\[
 \lambda_c^Q(t)\ge\frac1t>0.
\]

Irreducibility and \eqref{eq:multistep-bridge} imply
$Q(x,y)>0$ whenever $x\ne y$: otherwise every $Q^n(x,y)$ would vanish.  Fix
$o\in V$ and let

\[
 q_*:=\min_{y\sim o}\min\{Q(o,y),Q(y,o)\}>0.
\]

Equivariance and transitivity make $q_*$ a lower bound in both directions
along every edge of $G$.  Let $M_{x,y}$ be the number of frogs born at $x$
whose first jump occurs before time $t$ and lands at $y$.  Simultaneous
Poisson thinning over the possible first jumps makes the variables
$(M_{x,y})_{x\ne y}$ independent, with

\[
 M_{x,y}\sim\Poi\bigl(\lambda(1-e^{-t})Q(x,y)\bigr).
\]

Declare an edge $\{x,y\}$ open when $M_{x,y}\ge1$ and $M_{y,x}\ge1$.
Distinct bond events use disjoint thinning variables, so these bond variables
are independent and their parameters are at least

\[
 \left(1-\exp\{-\lambda(1-e^{-t})q_*\}\right)^2.
\]

Since $p_c(G)<1$ for every quasi-transitive graph of superlinear growth
\cite{DuminilCopinGoswamiRaoufiSeveroYadin2020}, choosing $\lambda$ large
enough makes this independent inhomogeneous percolation dominate a
supercritical homogeneous bond percolation: use one independent uniform
random variable per edge to couple the two models.  It therefore has an infinite
open cluster with positive probability.  An open path is an activation path,
so $\lambda_c^Q(t)<\infty$.

The kernel-independent form of
Lemma~\ref{lem:exit-density} and then
Lemma~\ref{lem:bridge-exit-bound} with $A=S^c$ give, pointwise in
$x\in S$,

\begin{equation}
 t q_{t,Q}^S(x)
 \le \E_x^Q\bigl[N_t\1_{\{\tau_{S^c}\le t\}}\bigr]
 \le B_\kappa(t)h_{t,Q}^S(x).
 \label{eq:bridge-exit-density}
\end{equation}

Substitute \eqref{eq:bridge-exit-density} into \eqref{eq:russo-time-Q} and use
\eqref{eq:russo-lambda-Q} to obtain

\begin{equation}
 \partial_t f_\Lambda^Q(\lambda,t)
 \le\frac{\lambda B_\kappa(t)}t
      \partial_\lambda f_\Lambda^Q(\lambda,t).
 \label{eq:bridge-reverse-differential}
\end{equation}

The coefficient $B_\kappa(t)/t$ depends only on time and is bounded on
compact positive time intervals.  Put

\begin{equation}
 A_\kappa(s,t):=\int_s^t\frac{B_\kappa(r)}r\di r,
 \qquad 0<s<t.
 \label{eq:bridge-characteristic-action}
\end{equation}

Apply \eqref{eq:bridge-reverse-differential} along
$\lambda_r=\lambda\exp\{-A_\kappa(s,r)\}$.  The chain rule gives
$\frac{\mathrm d}{\mathrm dr}f_\Lambda^Q(\lambda_r,r)\le0$; integrating and
then exhausting $V$ gives

\begin{equation}
 \theta_Q(\lambda e^{-A_\kappa(s,t)},t)
 \le\theta_Q(\lambda,s),
 \qquad 0<s<t.
 \label{eq:bridge-characteristic-comparison}
\end{equation}

Replacing \eqref{eq:reverse-characteristic-comparison} by
\eqref{eq:bridge-characteristic-comparison} in the proof of
Theorem~\ref{thm:continuous-critical-density} proves continuity of
$\lambda_c^Q$.
Strict decrease follows from Theorem~\ref{thm:strict-critical-density} and
$0<\lambda_c^Q(t)<\infty$.  If $0<s<t$ and
$\lambda<\lambda_c^Q(s)$, then
\eqref{eq:bridge-characteristic-comparison} gives
$\lambda e^{-A_\kappa(s,t)}\le\lambda_c^Q(t)$.  Letting
$\lambda\uparrow\lambda_c^Q(s)$ and using the forward comparison yields

\begin{equation}
 e^{-A_\kappa(s,t)}\lambda_c^Q(s)
 \le\lambda_c^Q(t)
 \le\frac{1-e^{-s}}{1-e^{-t}}\lambda_c^Q(s).
 \label{eq:bridge-critical-curve-two-sided}
\end{equation}
Equivalently,
\[
 \int_s^t\frac{\di r}{e^r-1}
 \le\log\frac{\lambda_c^Q(s)}{\lambda_c^Q(t)}
 \le\int_s^t\frac{B_\kappa(r)}r\di r.
\]

Both integrands are bounded above and bounded away from zero on compact
positive time intervals.  This proves local bi-Lipschitz regularity.
Repeating the proof of Proposition~\ref{prop:inverse-critical-continuity},
with $f=\lambda_c^Q$, using $\lambda_c^Q(t)\ge1/t$ and
\eqref{eq:bridge-critical-curve-two-sided}, gives the assertions for
$t_c^Q$.
\end{proof}

We next verify the two-step condition for the principal long-range example.

\begin{corollary}
\label{cor:heavy-tailed-critical-curves}
Let $G$ be a Cayley graph with $|B_G(o,n)|\asymp n^d$ for some $d>1$.
Let $Q$ be a symmetric Markov kernel invariant under left translations such
that, for some $\alpha>0$ and $0<c\le C<\infty$,

\begin{equation}
 c(1+d_G(x,y))^{-(d+\alpha)}
 \le Q(x,y)\le
 C(1+d_G(x,y))^{-(d+\alpha)},
 \qquad x\ne y.
 \label{eq:power-law-kernel}
\end{equation}

Then $\lambda_c^Q$ and $t_c^Q$ are continuous and strictly decreasing on
$(0,\infty)$, while $-\log\lambda_c^Q$ and $t_c^Q$ are locally
bi-Lipschitz there.  Moreover, all conclusions of
Corollary~\ref{cor:bridge-sharpness} hold; in particular, its explicit
supercritical survival bound and its exponential subcritical cluster tails
hold for this power-law model.
\end{corollary}

\begin{proof}
Put $s:=d+\alpha$ and

\[
 w(x,y):=(1+d_G(x,y))^{-s}.
\]

Polynomial growth and $s>d$ imply

\begin{equation}
 W:=\sum_{z\in V}w(o,z)<\infty,
 \qquad
 \sum_{z\in V}w(x,z)=W\quad(x\in V).
 \label{eq:power-law-summability}
\end{equation}

The finiteness of $W$ in \eqref{eq:power-law-summability} follows by summing
over the dyadic annuli $B(o,2^{k+1})\setminus B(o,2^k)$; invariance gives
the equality for every $x$.  For fixed $x,y,z\in V$, the triangle inequality
gives

\begin{equation}
 d_G(x,z)\ge\frac12d_G(x,y)
 \quad\text{or}\quad
 d_G(z,y)\ge\frac12d_G(x,y).
 \label{eq:power-law-distance-split}
\end{equation}

Splitting the convolution sum according to
\eqref{eq:power-law-distance-split} yields

\begin{equation}
 \sum_z w(x,z)w(z,y)
 \le 2^{s+1}Ww(x,y).
 \label{eq:power-law-subconvolution}
\end{equation}

Increase $C$ if necessary so that the upper bound in
\eqref{eq:power-law-kernel} also covers a possible diagonal mass; this is
possible because $Q(x,x)\le1=w(x,x)$.  Equations
\eqref{eq:power-law-kernel} and \eqref{eq:power-law-subconvolution} now give,
for $x\ne y$,

\begin{equation}
 Q^2(x,y)
 \le C^2 2^{s+1}Ww(x,y)
 \le\frac{C^2 2^{s+1}W}{c}Q(x,y).
 \label{eq:power-law-bridge-bound}
\end{equation}

Thus \eqref{eq:bridge-condition} holds with
$\kappa=C^2 2^{s+1}W/c$.  Theorem~\ref{thm:bridge-kernel-regularity} gives
the asserted regularity of $\lambda_c^Q$ and of $t_c^Q$ on
$(\lambda_\infty^Q,\infty)$.  Remark~1.6 and the discussion following the
proof of Theorem~1.5(2) in \cite{AngelDeLaRivaHermonShi2026} state that their
polynomial-growth renormalisation applies to the normalised nonholding
power-law kernel.  It therefore has finite critical lifespan at every
positive density.  For a possible holding mass, invariance makes
$Q(x,x)=\delta<1$ constant.  Deleting
holding jumps gives the kernel
$\overline Q(x,x)=0$ and
$\overline Q(x,y)=Q(x,y)/(1-\delta)$ for $x\ne y$, and merely changes the
clock rate from $1$ to $1-\delta$.  In particular,
\[
 \theta_Q(\lambda,t)
 =\theta_{\overline Q}(\lambda,(1-\delta)t),
 \qquad
 t_c^Q(\lambda)=\frac{t_c^{\overline Q}(\lambda)}{1-\delta}.
\]
Thus $t_c^Q(\lambda)<\infty$ for every $\lambda>0$,
$\lambda_\infty^Q=0$, and the generalised inverse is continuous and strictly
decreasing and locally bi-Lipschitz on $(0,\infty)$.
Corollary~\ref{cor:bridge-sharpness}, proved below, gives the remaining
sharpness conclusions.
\end{proof}

Corollary~\ref{cor:heavy-tailed-critical-curves} applies to the kernels in
\cite{AngelHermonShi2026HeavyTailed}, which studies discrete-time activation
times and the finite-volume cover lifespan rather than infinite-volume
critical curves.

Finally, the two-step estimate also supplies the input needed for the sharpness
argument of Angel et al.

\begin{corollary}
\label{cor:bridge-sharpness}
Under the hypotheses of
Theorem~\ref{thm:bridge-kernel-regularity}, for every $t>0$ and
$\lambda>\lambda_c^Q(t)$,
\begin{equation}
 \theta_Q(\lambda,t)
 \ge
 1-\left(
      \frac{\lambda_c^Q(t)}{\lambda}
    \right)^{1/(2B_\kappa(t))}.
 \label{eq:bridge-density-sharpness-strong}
\end{equation}
Consequently, for every $t>0$ there is $K_1>0$ such that

\[
 \theta_Q(\lambda,t)
 \ge K_1\bigl(\lambda-\lambda_c^Q(t)\bigr),
 \qquad
 \lambda_c^Q(t)\le\lambda\le2\lambda_c^Q(t),
\]

and for every $\lambda>\lambda_\infty^Q$, there is $K_2>0$ such that

\[
 \theta_Q(\lambda,t)
 \ge K_2\bigl(t-t_c^Q(\lambda)\bigr),
 \qquad
 t_c^Q(\lambda)\le t\le2t_c^Q(\lambda).
\]

Moreover, in either subcritical regime
$\lambda<\lambda_c^Q(t)$ or $t<t_c^Q(\lambda)$, there are constants
$C,c>0$ such that

\[
 \P_{\lambda,t,Q}(|\cC_o|\ge n)\le Ce^{-cn},
 \qquad n\ge0.
\]
\end{corollary}

\begin{proof}
For a finite $S\ni o$, put

\begin{equation}
 \begin{aligned}
  \phi_{\lambda,t,Q}(S)
  &:=\lambda\sum_{x\in S}
      a_{\lambda,t,Q}^S(x)h_{t,Q}^S(x),\\
  \widetilde\phi_{\lambda,t,Q}(S)
  &:=\lambda\sum_{x\in S}a_{\lambda,t,Q}^S(x)
      \E_x^Q\bigl[N_t\1_{\{\tau_{S^c}\le t\}}\bigr].
 \end{aligned}
 \label{eq:bridge-sharpness-quantities}
\end{equation}

Lemma~\ref{lem:bridge-exit-bound} gives, for every such $S$,

\begin{equation}
 \widetilde\phi_{\lambda,t,Q}(S)
 \le B_\kappa(t)\phi_{\lambda,t,Q}(S).
 \label{eq:bridge-sharpness-criterion}
\end{equation}

\begin{equation}
 c_\kappa(t):=\frac{1}{2B_\kappa(t)}>0.
 \label{eq:bridge-sharpness-threshold}
\end{equation}

The first quantity in \eqref{eq:bridge-sharpness-quantities} is the analogue
of \cite[Equation~(2.7)]{AngelDeLaRivaHermonShi2026}, and the second is the
analogue of \cite[Equation~(3.11)]{AngelDeLaRivaHermonShi2026}.
Proposition~\ref{prop:russo-all-kernels} supplies the density Russo identity
used in \cite[Section~3.1]{AngelDeLaRivaHermonShi2026}, while
\eqref{eq:bridge-sharpness-criterion} replaces the nearest-neighbour estimate
of Angel et al. in
\cite[Section~3.2]{AngelDeLaRivaHermonShi2026}.
Appendix~\ref{app:bridge-sharpness-adaptation} gives the adapted block
exploration, its refreshed-environment coupling, and the identification of
the resulting auxiliary threshold with $\lambda_c^Q(t)$.  It proves
\eqref{eq:bridge-density-sharpness-strong} and the exponential cluster tail
whenever $\lambda<\lambda_c^Q(t)$.  The exploration scheme itself is due to
Angel--de la Riva--Hermon--Shi; the new input here is
\eqref{eq:bridge-sharpness-criterion}, and the fixed choice
\eqref{eq:bridge-sharpness-threshold} yields the explicit global bound.

For $L:=\lambda_c^Q(t)$ and $L\le\lambda\le2L$, concavity gives
\[
 1-\left(\frac L\lambda\right)^{c_\kappa(t)}
 \ge\bigl(1-2^{-c_\kappa(t)}\bigr)\frac{\lambda-L}{L},
\]
which gives the asserted $K_1$.  For the lifespan conclusion, put
$T=t_c^Q(\lambda)$.  The forward critical-curve comparison gives
\[
 \frac{\lambda_c^Q(r)}{\lambda}
 \le\frac{1-e^{-T}}{1-e^{-r}},
 \qquad T<r\le2T.
\]
Inserting this into \eqref{eq:bridge-density-sharpness-strong}, uniformly
bounding $c_\kappa(r)$ away from zero on $[T,2T]$, gives the asserted $K_2$;
the calculation is recorded in
\eqref{eq:bridge-lifespan-sharpness-transfer} in the appendix.  Finally,
$t<t_c^Q(\lambda)$ implies $\lambda<\lambda_c^Q(t)$ by generalised
inversion, so the density-subcritical tail also gives the lifespan-subcritical
tail.
\end{proof}

Condition~\eqref{eq:bridge-condition} is sufficient but need not hold for an
arbitrary irreducible kernel.  For example, a finite-range kernel may have
$Q^2(x,y)>0=Q(x,y)$.  For simple random walk, the separate endpoint assertion
\eqref{eq:lambda-infinity-boundary} is precisely the lifespan-finiteness
part of \cite[Conjecture~1.1]{AngelDeLaRivaHermonShi2026}.

\appendix

\section{Adaptation of the sharpness proof to two-step kernels}
\label{app:bridge-sharpness-adaptation}

This appendix adapts the exploration and differential-inequality arguments
of \cite[Sections~3.1--3.2]{AngelDeLaRivaHermonShi2026}.  We include the
details because that argument is formulated for simple random walk, whereas
the present kernels may have infinite range. The new ingredients are the two-step exit estimate
\eqref{eq:bridge-exit-bound}, its fixed threshold
\eqref{eq:bridge-sharpness-threshold}, and the quantitative consequences
derived from them.

Throughout the appendix, assume the hypotheses of
Theorem~\ref{thm:bridge-kernel-regularity}.  Fix $t>0$, and retain the
quantities in \eqref{eq:bridge-sharpness-quantities}.  In the notation of
Angel, de la Riva, Hermon, and Shi, $\phi_{\lambda,t,Q}$ corresponds to their
$\phi_{\lambda,t}$ from \cite[Equation~(2.7)]{AngelDeLaRivaHermonShi2026},
while $\widetilde\phi_{\lambda,t,Q}$ corresponds to their
$\widetilde\phi_{\lambda,t}$ from
\cite[Equation~(3.11)]{AngelDeLaRivaHermonShi2026}.  Their bound in
\cite[Lemma~3.3]{AngelDeLaRivaHermonShi2026} is replaced, for every finite
$S\ni o$ and without a connectivity assumption, by
\eqref{eq:bridge-sharpness-criterion}.

\subsection*{The block offspring variable}

We first make the adapted subcritical exploration precise.  Fix a finite
set $S\ni o$.  For each $v\in V$, fix $\gamma_v$ in the transitive
equivariance group such that $\gamma_v o=v$, and put $S_v:=\gamma_v S$.
A fresh block rooted at $v$ has independent complete frog environments at
the sites of $S_v$, revealed by the following procedure.  Fix a deterministic
ordering of $V$ to resolve any ordering choices in the exploration.
\begin{enumerate}
 \item Initialise an internal queue with $v$, an empty set of processed
 sites, and an empty offspring list.
 \item Remove the first site $x$ from the internal queue.  If $x$ has
 already been processed in this block, discard this queue entry.  Otherwise,
 request the complete frog environment at $x$ and mark $x$ as processed.
 \item Classify each individual frog based at $x$ as staying or exiting,
 according to whether its complete time-$t$ trace stays in $S_v$.
 For every staying frog, add its unprocessed trace vertices to the internal
 queue.  For every exiting frog, append to the offspring list its location
 after each of its $N_t$ clock updates, retaining multiplicity, including
 repeated locations arising from holding updates.  These offspring entries
 are treated as roots of later blocks, even when they lie in $S_v$.
 \item Continue until the internal queue is empty.
\end{enumerate}
Only the staying frogs are used for internal activation.  At most $|S|$
site environments are requested in a block, so the procedure terminates
almost surely.  Listing every post-update location of an exiting frog
includes every vertex of its trace other than its already processed source.

Let $\xi$ be the length of the offspring list for a block rooted at $o$.
For $x\in S$, the event that $o$ internally activates $x$ has a simple
activation witness ending at $x$, and hence is measurable without using the
cloud based at $x$.  It is therefore independent of the marked collection of
exiting frogs based at $x$.  Poisson marking gives
\begin{equation}
 \E[\xi]
 \le \lambda\sum_{x\in S}a_{\lambda,t,Q}^S(x)
       \E_x^Q\bigl[N_t\1_{\{\tau_{S^c}\le t\}}\bigr]
 =\widetilde\phi_{\lambda,t,Q}(S).
 \label{eq:bridge-block-offspring-mean}
\end{equation}

Moreover, $\xi$ is at most the total number of clock updates made by all
frogs based in a complete fresh environment on the block.  The number of
frogs at its $|S|$ sites is Poisson with mean $\lambda|S|$, and each frog's
clock-update count $N_t$ has probability generating function
$\exp\{t(s-1)\}$.  Thus, for $s\ge1$,

\begin{equation}
 \E[s^\xi]
 \le
 \exp\!\left\{
   \lambda|S|\bigl(e^{t(s-1)}-1\bigr)
 \right\}<\infty.
 \label{eq:bridge-block-offspring-moment}
\end{equation}
In particular, $\xi<\infty$ almost surely.

\subsection*{Independent stacks and domination}

The refreshing construction of
\cite[Section~3.2]{AngelDeLaRivaHermonShi2026} can be made compatible with
overlapping long-range blocks as follows.

\begin{lemma}
\label{lem:refreshed-block-coupling}
The blocks can be coupled with the original frog environment so that their
offspring variables are independent with common law $\xi$.  If $T$ is the
total progeny of the resulting Galton--Watson process, then
\begin{equation}
 |\cC_o|\le |S|T.
 \label{eq:bridge-block-cluster-domination}
\end{equation}
\end{lemma}

\begin{proof}
We explore block roots in breadth-first order, beginning with $o$, and use the
internal-queue procedure above within each block.  At every site $x$, attach
a stack $(\mathcal E_{x,j})_{j\ge1}$ of complete frog environments.  All
stack levels are independent over $x$ and $j$ and have the original
site-environment laws, with $\mathcal E_{x,1}$ the original environment at
$x$.  After $n$ completed blocks, let $K_x(n)$ be the number of levels
requested at $x$, and let $\mathcal F_n$ be the sigma-field generated by the
revealed environments and the resulting block queue.  In the next block,
when $x$ is first processed by its internal queue, request
$\mathcal E_{x,K_x(n)+1}$.

By deferred decisions, conditionally on $\mathcal F_n$, the next unused
levels $(\mathcal E_{x,K_x(n)+1})_{x\in V}$ have their original independent
product law.  Indeed, every request is determined by previously revealed
levels, and an unused level has not entered any earlier exploration
decision.  The root of the next block is $\mathcal F_n$-measurable.  The
block algorithm, applied to the next unused levels on its translate $S_v$,
therefore has the same conditional law as a fresh block.  Equivariance makes
its offspring count have law $\xi$, independently of the preceding block
exploration.  If exploration terminates, extend the sequence by independent
variables with law $\xi$.  Thus the breadth-first block queue has the law
of a Galton--Watson exploration with offspring distribution $\xi$.

For domination, consider the following invariant after each completed
block: every vertex on an inspected level-one trace has either been
internally processed in some block or occurs as a root of a completed or
queued block.  Uninspected level-one environments retain their original
conditional product law.  To verify the invariant, we consider the first block
that requests the environment at a site $x$.  It uses level one and processes
$x$ internally.  Every vertex on a staying frog's trace is processed within
that block.  Every other trace vertex of an exiting frog appears among its
post-update locations and is put on the block queue; its source $x$ has
already been processed.  This accounts for every location on every
inspected level-one trace.

If $T=\infty$, \eqref{eq:bridge-block-cluster-domination} is immediate.
If $T<\infty$, the block queue eventually empties.  Let $D$ be the union
of the sites internally processed in all $T$ blocks.  It contains $o$.
For every $x\in D$, the first request at $x$ inspected its original
level-one environment.  The invariant and the empty final queue imply that
every vertex on every original frog trace based at $x$ also belongs to $D$:
each completed block root is internally processed in its own block.
Consequently, induction along any finite activation path in the original
environment gives $\cC_o\subseteq D$.  Each block processes at most $|S|$
sites, so $|\cC_o|\le |D|\le |S|T$, as claimed.
\end{proof}

If $\E[\xi]<1$, Equations~\eqref{eq:bridge-block-offspring-mean} and
\eqref{eq:bridge-block-offspring-moment} imply that $T$ has an exponential
moment.  For completeness, if $F(s)=\E[s^\xi]$, then
$F(1)=1$ and $F'(1)<1$, so there is $z>1$ with $F(z)<z$. Now choose
$s_0\in(1,z)$ so close to one that $s_0F(z)\le z$.  The recursive generating
functions for total progeny through finitely many generations are then
bounded by $z$ at $s_0$, and monotone convergence gives
$\E[s_0^T]\le z$.  Markov's inequality and
\eqref{eq:bridge-block-cluster-domination} therefore give
\begin{equation}
 \widetilde\phi_{\lambda,t,Q}(S)<1
 \quad\Longrightarrow\quad
 \P_{\lambda,t,Q}(|\cC_o|\ge n)\le Ce^{-cn},
 \qquad n\ge0,
 \label{eq:bridge-block-consequence}
\end{equation}
for constants $C,c>0$ depending on the parameters and $S$.

\subsection*{The auxiliary threshold}

We next adapt the threshold-identification argument of Angel et al., using
the fixed value $c_\kappa(t)$ from
\eqref{eq:bridge-sharpness-threshold}.  Define
\begin{equation}
 \widehat\lambda_c^Q(t)
 :=\sup\left\{\lambda\ge0:
   \inf_{\substack{S\ni o\\S\text{ finite}}}
   \phi_{\lambda,t,Q}(S)\le c_\kappa(t)
 \right\}.
 \label{eq:bridge-auxiliary-critical-density}
\end{equation}
The defining set is nonempty because it contains $\lambda=0$.
If $\lambda<\widehat\lambda_c^Q(t)$, the definition of the supremum gives a
$\mu>\lambda$ in the set on the right.  The infimum need not be attained, but
there is a finite $S\ni o$ such that
\[
 \phi_{\mu,t,Q}(S)<\frac{3}{4B_\kappa(t)}.
\]
Equations~\eqref{eq:bridge-sharpness-criterion} and
\eqref{eq:bridge-block-consequence} give an exponential cluster tail at
$(\mu,t)$.  Poisson thinning from $\mu$ to $\lambda$ preserves the bound with
the same constants.  Hence
\begin{equation}
 \lambda<\widehat\lambda_c^Q(t)
 \quad\Longrightarrow\quad
 \P_{\lambda,t,Q}(|\cC_o|\ge n)\le Ce^{-cn}
 \quad(n\ge0).
 \label{eq:bridge-auxiliary-subcritical}
\end{equation}
In particular,
$\widehat\lambda_c^Q(t)\le\lambda_c^Q(t)$; since the latter is finite,
$\widehat\lambda_c^Q(t)<\infty$.

Suppose now that $\lambda>\widehat\lambda_c^Q(t)$.  For every
$u\in(\widehat\lambda_c^Q(t),\lambda)$,
\[
 \inf_{\substack{S\ni o\\S\text{ finite}}}
 \phi_{u,t,Q}(S)>c_\kappa(t).
\]
For a finite $\Lambda\ni o$, Proposition~\ref{prop:russo-all-kernels} and
the identity
$\sum_{S\ni o}\P_{u,t,Q}(L_\Lambda^Q=S)=1-f_\Lambda^Q(u,t)$ give
\begin{equation}
 \partial_u f_\Lambda^Q(u,t)
 \ge\frac{c_\kappa(t)}u
      \bigl(1-f_\Lambda^Q(u,t)\bigr).
 \label{eq:bridge-auxiliary-russo-inequality}
\end{equation}
If $f_\Lambda^Q$ reaches one, the next bound is immediate.  Otherwise we divide
by $1-f_\Lambda^Q$ and integrate from
$a\in(\widehat\lambda_c^Q(t),\lambda)$ to $\lambda$ to obtain
\[
 f_\Lambda^Q(\lambda,t)
 \ge1-\left(\frac a\lambda\right)^{c_\kappa(t)}.
\]
Let $a\downarrow\widehat\lambda_c^Q(t)$, and exhaust the locally finite
reference graph by finite balls.  The exit events decrease to
$\{|\cC_o|=\infty\}$ even when $Q$ has infinite range: an infinite cluster
exits every ball, while every finite cluster is contained in some ball.
Therefore
\begin{equation}
 \theta_Q(\lambda,t)
 \ge1-\left(
   \frac{\widehat\lambda_c^Q(t)}\lambda
 \right)^{c_\kappa(t)},
 \qquad \lambda>\widehat\lambda_c^Q(t).
 \label{eq:bridge-auxiliary-supercritical}
\end{equation}
This is strictly positive for every
$\lambda>\widehat\lambda_c^Q(t)$, so
$\lambda_c^Q(t)\le\widehat\lambda_c^Q(t)$.  Together with the reverse
inequality above,
\begin{equation}
 \widehat\lambda_c^Q(t)=\lambda_c^Q(t).
 \label{eq:bridge-auxiliary-equals-critical}
\end{equation}
Equations~\eqref{eq:bridge-auxiliary-subcritical} and
\eqref{eq:bridge-auxiliary-equals-critical} prove the exponential tail for
every $\lambda<\lambda_c^Q(t)$, while
\eqref{eq:bridge-auxiliary-supercritical} becomes the explicit bound
\eqref{eq:bridge-density-sharpness-strong}.  This fixed-exponent global bound
is the quantitative gain from the new two-step estimate.

For completeness, the local density bound in
Corollary~\ref{cor:bridge-sharpness} follows directly.  If
$L=\lambda_c^Q(t)$ and $L\le\lambda\le2L$, the function
$x\mapsto1-(1+x)^{-c_\kappa(t)}$ is concave on $[0,1]$, so
\[
 \theta_Q(\lambda,t)
 \ge
 \bigl(1-2^{-c_\kappa(t)}\bigr)\frac{\lambda-L}{L}.
\]

\subsection*{Transfer to the lifespan parameter}

This last step is a consequence of the new global density bound and the
forward critical-curve comparison, rather than a new exploration.  Fix
$\lambda>\lambda_\infty^Q$ and put $T=t_c^Q(\lambda)$.  Then
$0<T<\infty$ and $\lambda_c^Q(T)=\lambda$.  For $T<r\le2T$,
\eqref{eq:critical-density-comparison} gives
\begin{equation}
 \frac{\lambda_c^Q(r)}\lambda
 \le \rho(T,r):=\frac{1-e^{-T}}{1-e^{-r}}.
 \label{eq:bridge-time-density-comparison}
\end{equation}
Set $c_*:=\min_{u\in[T,2T]}c_\kappa(u)>0$.  Strict decrease gives
$\lambda_c^Q(r)<\lambda$, so
\eqref{eq:bridge-density-sharpness-strong} and
\eqref{eq:bridge-time-density-comparison} yield
\begin{align}
 \theta_Q(\lambda,r)
 &\ge1-\rho(T,r)^{c_*}\notag\\
 &\ge\min\{1,c_*\}\bigl(1-\rho(T,r)\bigr)\notag\\
 &\ge\min\{1,c_*\}e^{-2T}(r-T),
 \qquad T<r\le2T.
 \label{eq:bridge-lifespan-sharpness-transfer}
\end{align}
For the last inequality, the mean-value theorem gives
$e^{-T}-e^{-r}\ge e^{-2T}(r-T)$, and the denominator in
$1-\rho(T,r)=(e^{-T}-e^{-r})/(1-e^{-r})$ is at most one.  The endpoint
$r=T$ is immediate, proving the claimed local lifespan bound.

Finally, generalised inversion and strict decrease of $\lambda_c^Q$ give,
for every $\lambda>0$,
\[
 t<t_c^Q(\lambda)\quad\Longrightarrow\quad
 \lambda<\lambda_c^Q(t).
\]
The density-subcritical estimate already proved therefore supplies the
exponential cluster tail in the lifespan-subcritical regime.

\paragraph{Funding.}
CM's research was partly funded by Deutsche Forschungsgemeinschaft
(DFG, German Research Foundation) under grant no.\ 443916008 in the
Priority Programme \emph{Random Geometric Systems} (SPP~2265).

\printbibliography

\end{document}